\documentclass{article}

\usepackage{microtype}
\usepackage{graphicx}
\usepackage[caption=false]{subfig}
\usepackage{booktabs} % for professional tables

\usepackage{hyperref}

\usepackage[accepted]{icml2020}

\usepackage{amsmath,amssymb, amsthm}
\usepackage{bbm}
\usepackage[capitalise,noabbrev]{cleveref}
\let\hat\widehat

\newtheorem{theorem}{Theorem}

\newtheorem{corollary}[theorem]{Corollary}

\newtheorem{proposition}[theorem]{Proposition}
\newtheorem{remark}[theorem]{Remark}
\newtheorem{definition}{Definition}

\usepackage{amsmath}
\DeclareMathOperator*{\argmax}{arg\,max}

\newcommand{\Tau}{\mathcal{T}}

\newcommand{\E}{\mbox{$\mathbb{E}$}}
\newcommand{\D}{\mbox{$\mathcal{D}$}}
\newcommand{\F}{\mbox{$\mathcal{F}$}}
\newcommand{\A}{\mbox{$\mathcal{A}$}}

\icmltitlerunning{On Conditional Versus Marginal Bias in Multi-Armed Bandits}

\begin{document}

\twocolumn[
\icmltitle{On Conditional Versus Marginal Bias in Multi-Armed Bandits}



\icmlsetsymbol{equal}{*}

\begin{icmlauthorlist}
\icmlauthor{Jaehyeok Shin}{cmu}
\icmlauthor{Alessandro Rinaldo}{cmu}
\icmlauthor{Aaditya Ramdas}{cmu,cmu2}
\end{icmlauthorlist}

\icmlaffiliation{cmu}{Department of Statistics and Data Science, Carnegie Mellon University}

\icmlaffiliation{cmu2}{Machine Learning Department, Carnegie Mellon University}

\icmlcorrespondingauthor{Jaehyeok Shin}{shinjaehyeok@cmu.edu}
\icmlkeywords{Conditional bias, Multi-armed bandit}

\vskip 0.3in
]



\printAffiliationsAndNotice{}  % leave blank if no need to mention equal contribution

\begin{abstract}
    The bias of the sample means of the arms in multi-armed bandits is an important issue in adaptive data analysis that has recently received considerable attention in the literature. Existing results relate in precise ways the sign and magnitude of the bias to various sources of data adaptivity, but do not apply to the conditional inference setting in which the sample means are computed only if some specific conditions are satisfied. In this paper, we characterize the sign of the conditional bias of monotone functions of the rewards, including the sample mean. Our results hold for arbitrary conditioning events and leverage natural monotonicity properties of the data collection policy. We further demonstrate, through several examples from sequential testing and best arm identification, that the sign of the conditional and marginal bias of the sample mean of an arm can be different, depending on the conditioning event. Our analysis offers new and interesting perspectives on the subtleties of assessing the bias in data adaptive settings.
\end{abstract}

\section{Introduction}
In modern data analysis, it is often the case that both the data collection and the selection of the target for statistical inference are determined adaptively, i.e. in a manner that depends on the realization of the data observed so far. In these fairly common data-adaptive scenarios, the validity of traditional procedures for statistical inference (that were designed to work in nonadaptive or in i.i.d. settings) can be severely affected. 

This issue is perhaps best exemplified within the stochastic multi-armed bandit (MAB) framework \citep{robbins1952some}. The data are collected sequentially in stages, during which the analyst draws a sample from one among finitely many unknown distributions or arms. At each stage, the determination of whether no more samples are to be acquired ({\it adaptive stopping}), the choice of the arm to draw from ({\it adaptive sampling}) and the selection of the target arm to analyze once the experiment has terminated ({\it adaptive choosing}) are all made based on previously observed data.  

The combined effect of these sources of data adaptivity will impact the correctness of standard, nonadaptive statistical procedures in ways that are generally difficult to discern and to quantify. 
 The recent literature on adaptive data analysis in MAB settings has focused primarily on the fundamental problem of mean estimation \citep{xu2013estimation, villar2015multi, bowden2017unbiased, nie2018adaptively, shin2019bias}.  It is now known that the sample mean of an arm is a biased estimator of the true mean and that the sign and magnitude of the bias can be related to adaptive sampling, stopping and choosing in precise ways.

The current theoretical understanding of the bias in MAB experiments is, however limited in two aspects. First, virtually all results concern the bias in mean estimation and do not cover other functionals of the arms. Secondly, and perhaps more interestingly, the existing guarantees cover only the  {\it marginal bias} of the sample mean, i.e., the bias obtained by accounting for all possible outcomes of the MAB experiment. However, in practice, one is often interested in the sample means of the arms only when certain outcomes have occurred. For instance,  the analyst may wish to evaluate the sample mean of a given arm only when that arm was identified as the best arm or, in a sequential framework corresponding to a MAB experiment with only one arm, when the null hypothesis has been rejected or when the random criterion for determining whether enough samples have been collected has been met. In all these cases, it is of interest to compute the {\it conditional bias}, i.e., the bias of the sample mean given a certain {\it conditioning
event,} such as that the arm of interest turned out to be the best arm. A priori, it is not at all clear how the sign of the conditional bias is affected by the choice of the conditioning event and by other sources of data adaptivity (e.g., sampling and stopping), or whether the signs of the marginal and conditional bias should match.  See \cref{fig:margianl_conditional_example} for illustrations of marginal and conditional biases. 

\begin{figure}%
    \centering
    \subfloat[Marginal bias]{{\includegraphics[scale =  0.24]{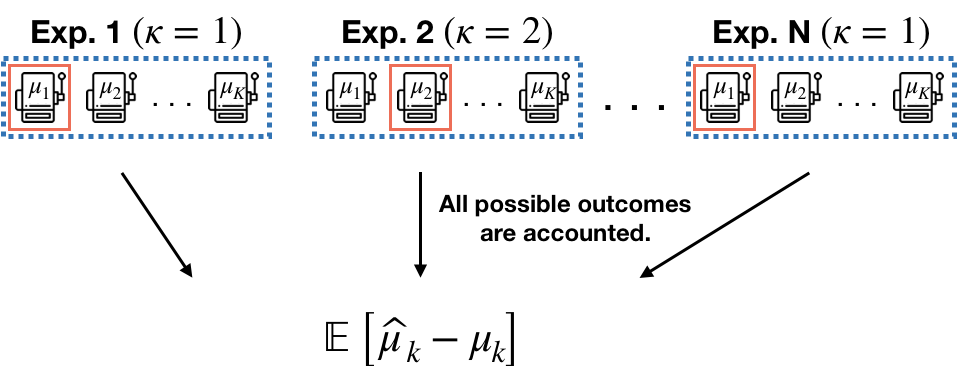} }}%
    \qquad
    \subfloat[Conditional bias]{{\includegraphics[scale =  0.24]{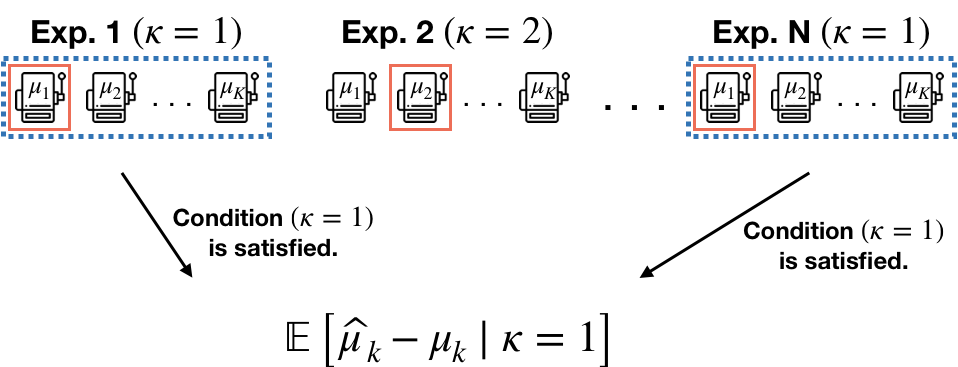} }}%
    \caption{ Illustrations of marginal and conditional biases in the $K$ prototypes example. (a) Marginal bias is obtained by accounting for all possible outcomes. (b) Conditional bias is obtained by accounting for outcomes related to a conditioning event.}%
    \label{fig:margianl_conditional_example}%
\end{figure}

As a concrete example, suppose we have $K$ prototypes of an online service and wish to test whether the potential average revenue of each prototype (i.e., arm), $\mu_k$, would be larger than a pre-specified threshold $\mu_0 >0$ based on a stream of test user samples. The usual sequential testing approach will be based on an appropriate upper stopping boundary and will reject each null $\mu_k \leq \mu_0$ if the corresponding sample mean has crossed such boundary during the testing period. If the null is rejected, then we conclude the prototype as a promising one. If not, we disregard it.
It is well known \citep[see, e.g.,][]{starr1968remarks, shin2019bias} that for each prototype, the sample mean based on data collected through this sequential testing procedure is positively biased: that is, for each $k \in [K]:=\{1,\ldots,K\}$, $\mathbb{E}\left[\hat{\mu}_k \right] \geq \mu_k$  regardless of whether the true mean $\mu_k$ is larger or smaller than the threshold $\mu_0$. This positive bias result can provide a useful warning signal about possible overestimation of the true revenues. 
At the same time, however, without careful consideration of the conditioning effect, we may end up with a false sense of comfort with low sample mean estimates from ``disregarded'' prototypes. Indeed,  we can naively expect the true revenues of the disregarded prototypes should be even lower than the observed estimates based on the positive bias result $\mathbb{E}\left[\hat{\mu}_k \right] \geq \mu_k$. In fact, this would be a wrong conclusion: conditioned on the disregarding event $C$, the sample mean is negatively biased and we have $\mathbb{E}\left[\hat{\mu}_k \mid C \right] \leq \mu_k$, as demonstrated in \cref{subSec::eg1}.

In this paper, we make the following contributions:
\begin{itemize}
    \item We derive in \Cref{thm::cond_bias} sufficient conditions for determining the sign of the conditional bias that hold under arbitrary conditioning events and rely on certain natural, highly interpretable monotonicity properties of the rules used for adaptive sampling, stopping and selecting. Our analysis captures in a rigorous and intuitive way the interaction between adaptivity arising from the data collection stage (sampling and stopping) with adaptivity in the choice of the target for inference. 
        \item We characterize the sign of the conditional bias of monotone functions of the rewards of each arm, which includes the sample mean as a special case.
    \item In \Cref{sec::simulations}, we demonstrate with several examples in best arm identification and sequential testing problems how the marginal and conditional bias of the sample means of the arms can have opposite signs. These are, we believe, instances of a  general, important phenomenon of theoretical and practical relevance. 
\end{itemize}
Overall, our results advance our ability to assess the impact of the bias in adaptive data analysis problems and offer several new and interesting insights on this important issue.

 The rest of this paper is organized as follows. In Section \ref{sec::acda_setting}, we briefly formalize the three notions of adaptivity by introducing a stochastic MAB framework. Section~\ref{sec::sign-bias} derives results on when the bias can be positive or negative. In Section~\ref{sec::simulations}, we demonstrate the correctness of our theoretical predictions through simulations in a variety of practical situations. We end with a brief summary in Section~\ref{sec::summary}, and for reasons of space, we defer all proofs to the Appendix.

% Explain multiple independent tests 

\section{The Stochastic Multi-Armed Bandit Model}
\label{sec::acda_setting}
In this section we formally describe the MAB experiment setting, detailing the sources of data adaptivity.  
% \textcolor{red}{Just copied from the previous paper. I will make it more concise.}
Let  $P_1, \dots, P_K$ be $K$ distributions of $K$ arms with finite means $\mu_k = \int x \mathrm{d}P_k(x)$.  If it is clear from context, every inequality and equality between two random variables is understood in the almost sure sense.

\subsection{Three Sources of Adaptivity in MABs}
A general stochastic MAB algorithm consists of the following sampling, stopping, and choosing rules:

\begin{itemize}
    \item {\bf Sampling:} For each time $t \geq 1$, choose a possibly randomized index of arm $A_t$ from a multinomial distribution $\mathrm{Multi}(1, \nu_t) \in \{0,1, \dots, K\}$. Then, we draw a sample (also called a reward) $Y_t$ from the chosen arm $P_{A_t}$. Here, $\nu_t \in [0,1]^K$ refers to conditional probabilities of sampling each arm defined by
    \[
     \nu_t(k):= \nu_t(k \mid \D_{t-1}) \in [0,1]~~\forall k \in [K],
    \]
    with the constraint $\sum_{k=1}^K \nu_t(k) = 1$,
    and $\D_{t-1}$ refers to observed data up to time $t-1$ including all external randomness. If each $\nu_t$ does not depend on the previous data $\D_{t-1}$, we call it a nonadaptive sampling rule. 
    
    With a proper choice of $\nu_t$, this sampling rule captures a broad class of existing methods including fixed design, random allocation, $\epsilon$-greedy \citep{sutton1998introduction}, upper confidence bound (UCB) algorithms \citep{auer2002finite, audibert2009minimax, garivier2011kl,kalyanakrishnan2012pac, jamieson_lil_2014} and Thompson sampling \citep{thompson1933likelihood, agrawal2012analysis,kaufmann2012thompson}. 
    \item {\bf Stopping:} Given a sampling rule, let $\{\mathcal{F}_t\}$ be a filtration such that each $A_t$ is $\F_{t-1}$-measurable and each $Y_t$ and $\D_t$ are $\F_t$-measurable.     Let $\Tau$ be a stopping time with respect to the filtration $\{\mathcal{F}_t\}$. Then, for each time $t$, the stopping event $\{\Tau = t\}$ is $\F_t$-measurable which can be used to characterize a stopping rule of a MAB algorithm.
    
    For example, the following stopping time corresponds to the stopping rule in which we stop whenever the absolute difference between sample means for arm $1$ and arm $2$ exceeding a fixed threshold $\delta >0$:
    \[
    \Tau := \inf\left\{t \geq 1: \left|\hat{\mu}_{1}(t) - \hat{\mu}_{2} (t)\right| > \delta  \right\}.
    \]
    Here, we refer $\hat{\mu}_k(t)$ to the sample mean for arm $k$ defined as
\[
\hat{\mu}_k(t) := \frac{S_k(t)}{N_k(t)},
\]
    where $S_k(t)$ and $N_k(t)$ are the running sum and number of draws for arm $k$ defined respectively $S_k(t) := \sum_{s=1}^{t} \mathbbm{1}(A_s = k) Y_s$ and $N_k(t) := \sum_{s=1}^{t} \mathbbm{1}(A_s = k)$ for each $k \in [K]$ and $t \geq 1$.\footnote{Throughout this paper, we assume $N_k(t) \geq 1$ whenever we make a statement about the sample mean $\hat{\mu}_k(t)$.} Note that, by the construction, $S_k(t)$, $\hat{\mu}_k(t)$ are $\F_t$-measurable and $N_k(t)$ is $\F_{t-1}$-measurable for each $k \in [K]$ and $t \geq 1$.
    
    If a stopping rule is nonadaptive, that is, if the corresponding stopping time is $\F_0$-measurable, we use $T$ instead of $\Tau$ to refer the stopping time which includes all fixed time and, more generally, all data-independent stopping rules. 
\item {\bf Choosing:}  After stopping, let $\D_\Tau$ denote collected data up to the stopping time $\Tau$. Based on $\D_{\Tau}$ with possible randomization, we choose a data-dependent arm based on a choosing rule $\kappa : \mathcal{D}_\Tau \mapsto [K]$. In this paper, we denote the index of chosen arm $\kappa(\mathcal{D}_\Tau)$ by $\kappa$ for simplicity. 

For example, in the best arm identification, we often choose the best arm as the arm with the largest sample mean at the stopping time $\Tau$. In this case, the choosing rule is given as $\kappa = \argmax_{k\in[K]} \hat{\mu}_k(\Tau)$. If $\kappa$ does not depend on the observed data, we call it a nonadaptive choosing rule. 
\end{itemize}

%  The phrase ``fully adaptive setting'' refers to the scenario of running an adaptive sampling algorithm until an adaptive stopping time $\Tau$ and asking about the sample mean of an adaptively chosen arm $\kappa$. When we are not in the fully adaptive setting, we explicitly mention what aspects are adaptive. 

\subsection{The Tabular Model of Stochastic MABs} \label{sec::counterfactual_setting}
% The tabular perspective on stochastic MABs has been used either internally or explicitly 
To analyze the theoretical properties of a MAB experiment, it is convenient to express it using an infinite tabular representation of the arm rewards ( Section 2.2 of \citet{shin2019bias}, \citet[Section 4.6]{lattimore2018bandit}).

In the tabular model, we assume that each observation $Y_t$ comes from an imaginary $\mathbb{N} \times K$ table or stacks of samples $ \{X_{i,k}^*\}_{i \in \mathbb{N}, k \in [K]}:= X_{\infty}^*$ where each $X_{i,k}^*$ is an independent draw from $P_k$, and the observation $Y_t$ is equal to $X_{i,k}^*$ if $N_k(t) = i$ and $A_t = k$ .

Now, given a MAB algorithm, let $W_t$ be a random variable which  accounts for  all possible external randomness of the algorithm for each time $t$. Next, we define $\D^*_\infty =X^*_{\infty} \cup \{W_{t}\}_{t=0}^\infty$ as the collection of all possible arm rewards and external randomness.
%a ``scenario'' of the MAB algorthm by the collection of the stacks of samples and all external randomness denoted by $\D^*_\infty =X^*_{\infty} \cup \{W_{t}\}_{t=0}^\infty$. 
Then, given $\D_{\infty}^*$, every property and outcome of a MAB experiment becomes a deterministic function of $\D_{\infty}^*$. In particular  $\Tau$, $\kappa$ and $N_k(t)$, for each $t$ and $k$, can be expressed as deterministic functions of $\D_{\infty}^*$. See Section 2.2 in \citet{shin2019bias} for details. 

\section{Monotonicity Property and Signs of Biases} \label{sec::sign-bias}

In a recent work by \citet{shin2019bias}, a notion of monotonicity of the data collecting was introduced to characterize the marginal bias of a chosen sample mean at a stopping time. In this section, we summarize the previous result with some minor modifications in notations.

\begin{definition} \label{def::monotone_sampling_strategy}
    For each $k \in [K]$, we say a MAB algorithm satisfies the ``monotonic increasing (or decreasing) property for arm $k$'' if for any $i \in \mathbb{N}$, the function  $\D^*_\infty \mapsto   \mathbbm{1}\left(\kappa = k\right) / N_k(\Tau)$, is an increasing (or decreasing) function of $X_{i,k}^*$  while keeping all other entries in $\D^*_\infty$ fixed.     Further, we say that
    \begin{itemize}
        \item A sampling rule is optimistic for arm $k$ if the function  $\D^*_\infty \mapsto  N_k(t)$ is an increasing function of $X_{i,k}^*$  while keeping all other entries in $\D^*_\infty$ fixed for any fixed $i \in \mathbb{N}$ and $t \geq 1$;
        \item  A stopping rule is optimistic for arm $k$ if the function $\D^*_\infty \mapsto  \Tau$ is a decreasing function of $X_{i,k}^*$ while keeping all other entries in $\D^*_\infty$ fixed for any fixed $i \in \mathbb{N}$;
        \item  A choosing rule is optimistic for arm $k$ if the function $\D^*_\infty\mapsto\mathbbm{1}(\kappa =k)$ is an increasing function of $X_{i,k}^*$ while keeping all other entries in $\D^*_\infty$ fixed for any fixed $i \in \mathbb{N}$.
    \end{itemize}
\end{definition}

It can be shown that many MAB algorithms, including $\epsilon$-greedy, UCB, and Thompson sampling algorithms, have optimistic sampling rules. As for optimistic stopping, a sequential testing procedure with an upper stopping boundary yields a typical example of an optimistic stopping rule based on the stopping time $\Tau := \inf\left\{t: \hat{\mu}(t) > U(t)\right\}$ where $t \mapsto U(t)$ is a deterministic function characterizing the upper stopping boundary. Finally, the index of the arm with the largest sample mean, $\kappa = \argmax_{k \in [K]} \hat{\mu}_k(\Tau)$ can be viewed as a typical optimistic choosing rule.

 \subsection{The Sign of the Marginal Bias}
 \citet{shin2019bias} demonstrated how certain natural monotonicity properties of the MAB experiment can be exploited to determine the sign of the marginal bias and of the conditional bias given conditioning events corresponding to adaptive choosing of arms. For completeness, we recall this result next.
%The following proposition, which was originally presented in shows how the can be used to determine the sign of the bias of the sample mean conditioned on the choosing event and the sign of the marginal bias.

\begin{proposition}[Theorem~7 in \citet{shin2019bias}] \label{prop::marginal_bias}
    Given a MAB algorithm, suppose we have $\mathbb{E}\left[N_k(\Tau)\right] < \infty$ and $N_k(\Tau) \geq 1$ for all $k$ with $\mathbb{P}\left(\kappa =k\right) > 0$. If the MAB algorithm satisfies the monotonic decreasing property for each arm $k$ with $\mathbb{P}\left(\kappa =k\right) > 0$ then we have 
    \begin{equation} \label{eq::sign_of_cov_bias_increasing_each_k}
    \E\left[ \hat{\mu}_k (\Tau)\mid \kappa = k \right] \leq \mu_k,
    \end{equation}
which also also implies that
    \begin{equation}\label{eq::sign_of_cov_bias_increasing}
    \E\left[ \hat{\mu}_\kappa(\Tau) - \mu_\kappa \right] \leq 0.  \end{equation} 
    Similarly if the MAB algorithm satisfies the monotonic increasing property for each arm $k$ with $\mathbb{P}\left(\kappa =k\right) > 0$ then we have 
    \begin{equation} \label{eq::sign_of_cov_bias_decreasing_each_k}
    \E\left[ \hat{\mu}_k(\Tau)\mid \kappa = k \right] \geq \mu_k,
    \end{equation}
    which also implies that
    \begin{equation}\label{eq::sign_of_cov_bias_decreasing}
    \E\left[ \hat{\mu}_\kappa(\Tau) - \mu_\kappa \right] \geq 0. 
    \end{equation} 
    If each arm has a bounded distribution then the condition $\mathbb{E}\left[N_k(\Tau)\right] < \infty$ can be dropped. \end{proposition}

Note that if a MAB algorithm consists of an optimistic sampling but nonadaptive stopping and choosing rules, then the algorithm satisfies the monotonic decreasing property. This is the reason why $\epsilon$-greedy, UCB, and Thompson sampling algorithms yield a negatively biased sample mean for each arm, marginally. 

Also, if a MAB  algorithm consists of an optimistic stopping but nonadaptive sampling and choosing, the algorithm satisfies the monotonic increasing property, which explains why sequential testing procedures (corresponding to one-arm MBA experiments) with an upper stopping boundary produces a positively biased sample mean.

\subsection{The Sign of the Conditional Bias}
%\textcolor{red}{Update from here...}

In this section, we generalize \cref{prop::marginal_bias} and derive sufficient conditions for the sign of the conditional bias under arbitrary conditioning events of monotone functions of the arms rewards.  

%CDFs and monotone functionals conditioned not only on the choosing event but on more general selection events.
% \textcolor{red}{Change the definition of CDF and distribution itself. See photo.}

We first introduce some  notation. For each arm $k$, let $F_k$ denote the corresponding cumulative distribution function (CDF): $y \in \mathbb{R} \mapsto F_k(y) = \mathbb{P}( X \leq y),$ where $X \sim P_k$.
Let $\hat{F}_{k, t}$ be the empirical CDF for arm $k$ based on samples up to time $t$; that is, $\hat{F}_{k, t}$ is a random function on $\mathbb{R}$ and taking values in $[0,1]$ given by 
\begin{equation}
  y \in \mathbb{R} \mapsto   \hat{F}_{k,t}(y):= \frac{1}{N_k(t)}\sum_{s=1}^t \mathbbm{1}\left(A_s = k, Y_s \leq y \right),
\end{equation}
which returns the fractions of samples from arm $k$ whose values are no larger than $y$.
For a stopping time $\Tau$ with respect to the MAB filtration $\{\mathcal{F}_t\}$, we then define $\hat{F}_{k, \Tau}$ to be the  empirical CDF of the rewards of arm $k$ up to time $\Tau$; that is, for each $y \in \mathbb{R}$,
\[
\hat{F}_{k, \Tau}(y) := \lim_{t\to\infty}\hat{F}_{k, \Tau \wedge t}(y).
\]
The use of the limit in the above definition is a necessary technicality in order to allow for the possibility that $\Tau = \infty$.
Note that, on the event $\{ N_k(\Tau) = \infty \} = \{\lim_{t \to \infty} N_k(\Tau \wedge t) = \infty\}$, we have that, for each $y \in \mathbb{R}$, $\hat{F}_{k}(y) = F_k(y)$ almost surely, based on the strong law of large numbers; see  Theorem 2.1 in \citet{gut2009stopped}.

Next, for any  function $f: \mathbb{R} \to \mathbb{R}$ integrable with respect to $P_k$, let  $E_k f = \int f(x) \mathrm{d}P_k(x)$
be the corresponding expectation. Similarly, we let
\[
\hat{E}_{k,t} f =  \frac{1}{N_{k}(t)} \sum_{s=1}^t f(Y_s) \mathbbm{1}\left(A_s = k\right)  
\]
to be the expectation of $f$ under the empirical measure of the $k$-th arm at time $t$. Clearly,  $\hat{E}_{k,t} f$ is a random variable.  If $f$ is the identity function, then it is immediate to see that $E_k f = \mu_k$ and $\hat{E}_{k,t} f = \hat{\mu}_k(t)$. Also, for any $y \in \mathbb{R}$, setting $f$ to be the indicator function $x \mapsto f(x) = \mathbbm{1}(x \leq y )$ yields that $E_k f = F_k(y)$ and $\hat{E}_{k,t} f = \hat{F}_{k,t}(y)$. 
 Finally, for a (possibly infinite) stopping time $\Tau$ with respect to the filtration $\{ \mathcal{F}_t\}$, we set
 \[
 \hat{E}_{k,\Tau} f = \lim_{t \rightarrow \infty} \hat{E}_{k,  \Tau \wedge t} f.
 \]
%$\hat{E}_k f(X)$ be the expectation of $f(X)$ under the empirical measure and $E_k f(X)$ be the expectation of $f(X)$ under the underlying distribution of arm $k$. Note that $\hat{E}_k f(X)$ is a random variable, and if $f$ is the identity function, $\hat{E}_k f(X)$ is equal to the sample mean of arm $k$. Also $E_k f(X)$ is equal to the true mean $\mu_k$ of arm $k$.

% From the strong law of large numbers, for any fixed $y \in \mathbb{R}$, we have $ P_{N_k(\Tau)}\mathbbm{1}(N_k(\Tau) = \infty) = P_k(y)\mathbbm{1}(N_k(\Tau) = \infty)$ almost surely,\footnote{Technically, we define $P_{N_k(\Tau)}(y) := \lim_{t\to\infty}P_{N_k(\Tau \wedge t)}(y)$. Then, on the event $(N_k(\Tau) = \infty)$, the strong law of large numbers implies that $P_{N_k(\Tau)}(y) = P_k(y)$ almost surely on the event.}  which also implies that
% \begin{equation}\label{eq::expecation_on_infty}
%     \mathbb{E}\left[P_{N_k(\Tau)}\mathbbm{1}(N_k(\Tau) = \infty)\right] = P_k(y) \mathbb{P}\left(N_k(\Tau) = \infty\right).
% \end{equation}
% Therefore, the sign of $  \E\left[P_{N_{k}(\Tau)}(y) \mid C \right] - P_k (y)$

%The following theorem provides a sufficient condition to characterize whether the conditional expectations of $\hat{F}_k (y)$ and $\hat{E}_k f(X)$ are positively or negatively biased from their population counterparts $F_k(y)$ and $E_k f(X)$. 
We are now ready to present the main result of this paper. The proof is given in \cref{Appen::proof_of_cond_bias}. For any event $C$, we let $\mathbb{E}[ \cdot |C]$ denote the conditional expectation given $C$. 

\begin{theorem}
\label{thm::cond_bias}
   Let $\Tau$ be a stopping time with respect to the MAB experiment natural filtration $\{ \mathcal{F}_t\}$. For a fixed $k \in [K]$, suppose $N_k(\Tau) \geq 1$. Let $C \in \mathcal{F}_{\Tau}$ be any event at stopping time $\Tau$ with $\mathbb{P}(C) > 0$. Assume that, for each $i$, the function $\D^*_\infty \mapsto  \mathbbm{1}\left(C\right) / N_k(\Tau)$ is a decreasing function of $X_{i,k}^*$ while keeping all other entries in $\D^*_\infty$ fixed.  Then, for each $y \in \mathbb{R}$,
    \begin{equation} \label{eq::sign_of_cond_measure_pos}
        \inf_{y \in \mathbb{R}}\left(\E\left[\hat{F}_{k, \Tau}(y) \mid C \right] -  F_k (y)\right) \geq 0,
    \end{equation}
   or, equivalently, for any non-decreasing function $f:\mathbb{R} \to \mathbb{R}$ with $\E\left[\hat{E}_{k,\Tau} |f|\mid C\right] < \infty$ and $E_{k,\Tau} |f| < \infty$,
    \begin{equation} \label{eq::sign_of_conditional_bias_decreas}
        \E\left[\hat{E}_{k,\Tau} f  \mid C \right] \leq  E_k f.
    \end{equation}
    Similarly, if, for each $i$, the function $\D^*_\infty \mapsto  \mathbbm{1}\left(C\right) / N_k(\Tau)$ is an increasing function of $X_{i,k}^*$ while keeping all other entries in $\D^*_\infty$ fixed, then we have
    \begin{equation} \label{eq::sign_of_cond_measure_neg}
        \sup_{y \in \mathbb{R}}\left(\E\left[\hat{F}_{k,\Tau}(y) \mid C \right] -  F_k (y)\right) \leq 0,
    \end{equation} 
    % \textcolor{red}{should be $\sup$ and not $\inf$}
    or, equivalently,
    \begin{equation} \label{eq::sign_of_conditional_bias_increas}
          \E\left[\hat{E}_{k,\Tau} f  \mid C \right] \geq  E_k f,
    \end{equation}
    for any non-decreasing function $f:\mathbb{R}\to\mathbb{R}$ satisfying $\E\left[\hat{E}_{k,\Tau} |f|\mid C\right] < \infty$ and $E_{k,\Tau} |f| < \infty$.
\end{theorem}

%\begin{remark}
%In the  theorem, we have used the fact -- shown in the proof -- that, if two probability measures $P$ and $Q$ with corresponding CDF's $F_P$ and $F_Q$ are such that 
%\[
%F_P(y) \leq F_Q(y), \quad \forall y \in \mathbb{R},
%\]
%then, for any non-decreasing function $f$, $\mathbb{E}_{X \sim P}[f(X)] \geq \mathbb{E}_{X \sim Q}[f(X)]$, provided both expectations exist. 
%\end{remark}

The important conclusion from the above theorem is that the conditional expected value of the empirical  CDF of the rewards of any given arm can be stochastically smaller or greater than the true CDF of the corresponding arm.
And furthermore, the sign of the bias can be determined in the basis of natural and often verifiable (as shown below) monotonicity conditions that depend on (i) the specific sampling and stopping rules deployed in the MAB experiment and (ii) the choice of the conditioning event $C$. 

 Note that if in the theorem  we choose $C$ as the conditioning event $\{\kappa = k\}$ and the function $f$ as the identity function, \cref{thm::cond_bias} yields  \cref{prop::marginal_bias} of \citet{shin2019bias} as a special case (indeed the statements about the marginal bias in that result follows from those on the conditional bias given the conditioning event $\{ \kappa = k \}$).
We emphasize that \cref{thm::cond_bias} is a generalization of \cref{prop::marginal_bias} of \citet{shin2019bias} in at least two ways: (i) it provides a more general guarantee about the conditional bias of monotone functions of the empirical CDF of the arms as opposed to just the sample mean and (ii) it allows for virtually any conditioning event that depends on the outcome of the MAB experiment (formally, that is measurable with respect to $\mathcal{F}_{\Tau})$. 

The monotonicity assumption in \cref{thm::cond_bias} captures in a mathematically concise yet intuitive manner how adaptivity in the data collection process combines with adaptivity in the selective data analysis, exemplified by the conditioning event, to affect the sign of the bias. The interaction between these two sources of adaptivity is, at least to us, not at all apparent and, in fact, rather subtle. To illustrate this phenomenon, in the next section, we apply \cref{thm::cond_bias} in several, fairly routine,  situations in adaptive data analysis to demonstrates how marginal and conditional biases may very well have opposite signs. %These example should  serve as cautionary tales 

%\textcolor{red}{CDF version + figure}
%\textcolor{red}{Discuss Open question. (and think about new notion of optimism for variance or median) }
\section{Applications}\label{sec::simulations}
In this section, we discuss several practical examples of the conditional bias results in \cref{thm::cond_bias}.

% \subsection{Sequential likelihood ratio test for a signle arm}
% Suppose we a sub-Gaussian arm with known parameter $\sigma^2$ but unknown means $\mu$. In this subsection, we consider the following testing problem:
% \begin{equation}
% H_0 : \mu \leq \mu_0 ~~\text{vs}~~H_1: \mu >\mu_0.
% \end{equation}

\subsection{Conditional Versus Marginal Bias of a Stopped Sequential Test} \label{subSec::eg1}

Suppose we have a stream of samples from a single arm with a finite mean $\mu$. For each $t$, let $\hat{\mu}(t)$ be the sample mean of the rewards observed up to time $t$.

To test whether $\mu$ is larger than a reference value $\mu_0$, we may construct an upper boundary $t \mapsto U(t)$ and conclude that $\mu \geq \mu_0$ if the sample mean ever crosses the upper boundary. 

Let $\Tau$ be the first time the sample mean crosses the boundary, i.e.  $\Tau := \inf\left\{t \geq 1: \hat{\mu}(t) \geq U(t) \right\}$. The stopping time $\Tau$ is an example of an optimistic stopping and, from \cref{prop::marginal_bias}, we can check that the stopped sample mean $\hat{\mu}(\Tau)$ is always positively biased \citep{shin2019bias}.

However, the stopping time $\Tau$ can be large or even infinite with non-zero probability. Thus, in practice, we may want to allow for the possibility of stopping the sequential test before reaching the stopping time. Let $M$ be any fixed predefined time at which we stop the testing procedure (if still ongoing), and let $\Tau_M := \min\left\{\Tau, M \right\}$ be the corresponding stopping time which takes account the early stopping option. Again, by \cref{prop::marginal_bias}, we know that the stopped sample mean $\hat{\mu}(\Tau_M)$ is still positively biased, i.e. $\mathbb{E}\left[\hat{\mu}(\Tau_M)\right] \geq \mu,$
since the function $\D_{\infty}^* \mapsto 1 / \Tau_M$ is an increasing function of $X_i^*$ while keeping all other entries in $\D_{\infty}^*$ fixed. From \cref{thm::cond_bias}, we can also check that the expected empirical CDF at stopping time $\Tau_M$ is negatively biased: 
\begin{equation} \label{eq::eg1_marginal}
   \sup_{y \in \mathbb{R}}\left( \E\left[\hat{F}_{k, \Tau_M} (y) \right] - F_k(y) \right) \leq 0. 
\end{equation}
Conditioned on the early stopping event $\{M < \Tau\}$, however, \cref{thm::cond_bias} shows that the early stopped sample mean and the empirical CDF are negatively and positively biased, respectively, since the function $\D_{\infty}^* \mapsto \frac{\mathbbm{1}(M < \Tau)}{\Tau_M} = \frac{\mathbbm{1}(M < \Tau)}{M}$ is an decreasing function of $X_i^*$ while keeping all other entries in $\D_{\infty}^*$ fixed. That is,
\begin{align}
\mathbb{E}\left[\hat{\mu}(M) \mid M < \Tau \right] &\leq \mu, \\
\inf_{y \in \mathbb{R}}\left( \E\left[\hat{F}_{k, M} (y)\mid M < \Tau \right] - F_k(y) \right)& \geq 0.
\end{align}

However, we can also check that the function $\D_{\infty}^* \mapsto \frac{\mathbbm{1}(M \geq \Tau)}{\Tau_M} = \frac{\mathbbm{1}(M \geq \Tau)}{\Tau}$ is an increasing function of $X_i^*$, which implies that, on the line-crossing event $\left\{M \geq \Tau\right\}$, the sample mean and empirical CDF are positively and negatively biased, respectively. Thus, depending on the conditioning event, the conditional bias can be positive or negative. Also note that, without the early stopping condition, $\hat{\mu}(M)$ is an unbiased estimator of $\mu$.

\begin{figure}
    \begin{center}
    \includegraphics[scale =  0.65]{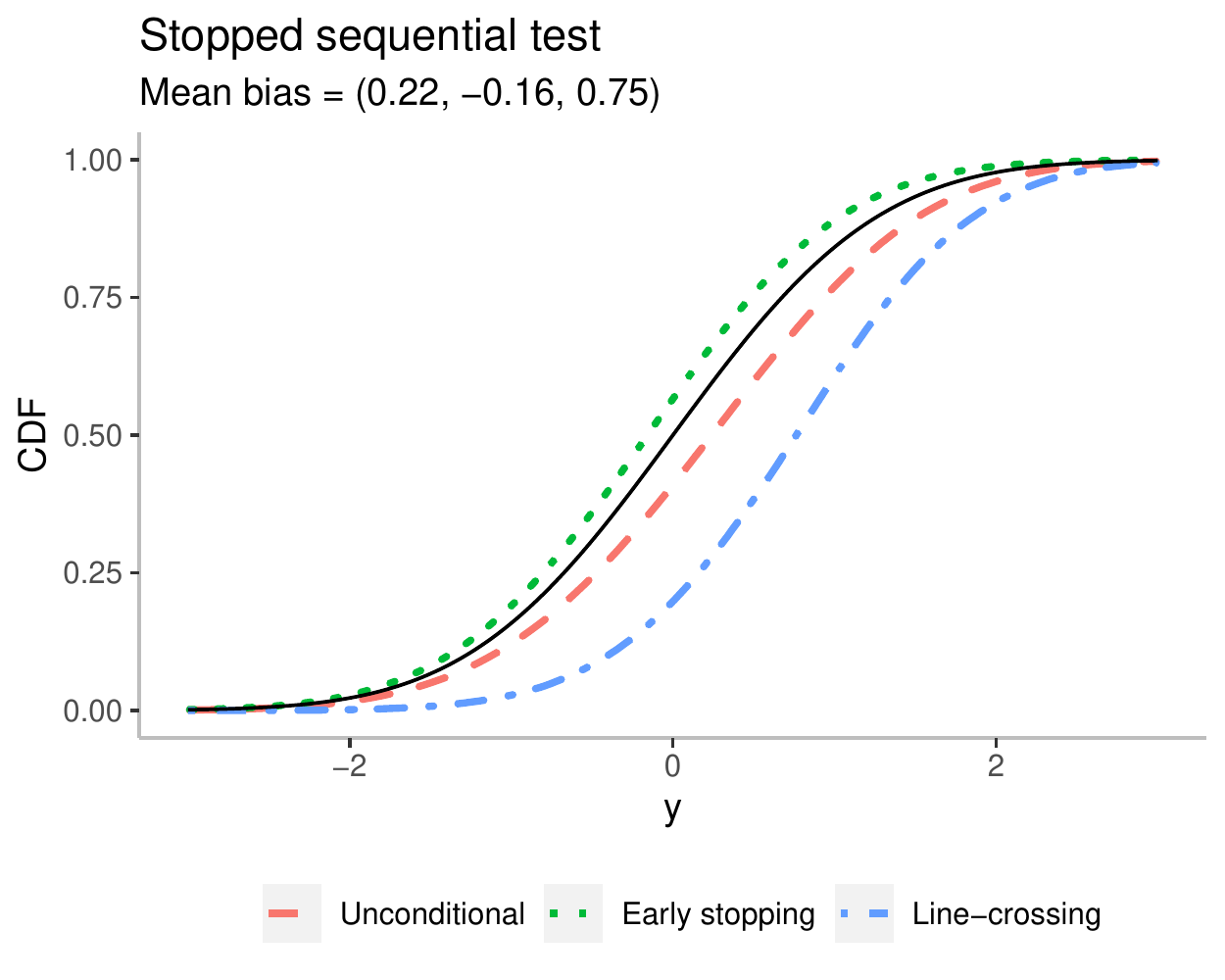}
    \end{center}
    \caption{\em  Average of marginal (red dashed line) and conditional empirical CDFs (green dotted and blue dot-dashed lines) from repeated stopped sequential test in \cref{subSec::eg1}.  The black solid line refers to the true CDF of the underlying arm}% The black solid line refers to the true CDF of the underlying arm. Marginally and conditionally on the line-crossing event, empirical CDFs are negatively biased (red dashed and blue dot-dashed lines). On the other hand, conditioned on the early stopping event, the empirical CDF is positively biased (green dotted line).}
    \label{fig::example1}
\end{figure}

{\bf Experiment:} We verify these facts with simulations, where we repeat a stopped sequential test $10^5$ times. In each test, the arm corresponds to a standard normal distribution, and the upper boundary is constructed by the point-wise upper confidence bounds  $t \mapsto U(t) = \frac{z_\alpha}{\sqrt{t}}$, where $z_\alpha$ is the $\alpha$-upper quantile of the standard normal distribution. Note that this upper boundary does not yield a valid testing procedure since it inflates the type 1 error. However, we choose this boundary with unusual parameters $\alpha = 0.2$ and $M = 10$ to manifest the difference between marginal and conditional biases. 

\cref{fig::example1} shows the point-wise averages of the observed marginal and conditional CDFs from the simulated $10^5$ stopped sequential tests. The black solid line refers to the true CDF of the underlying arm. The red dashed line corresponds to the marginal CDF, which lies below the true CDF, as expected, and thus the corresponding mean bias is positive. The green dotted and blue dot-dashed lines correspond to the averages of empirical CDFs conditioned on the early stopping and line-crossing events, respectively. As predicted, conditioned on the early stopping event, the empirical CDF and the sample mean are positively and negatively biased, respectively. In contrast, conditioned on the line-crossing event, the empirical CDF and the sample mean are negatively and positively biased, respectively.

% In \citet{starr1968remarks}, 

% , it was proved that if we stop when the sample mean crosses a predetermined upper boundary defined by $t \mapsto U(t)$ then the stopped sample mean 

% In \citet{starr1968remarks} 

% Consider the following testing problem:

% \begin{figure}[h]
%     \begin{center}
%     \includegraphics[scale =  0.65]{./Figures/example2_alternative.pdf}
%     \end{center}
%     \caption{\em  Contents...}
%     \label{fig::example1}
% \end{figure}
\subsection{Sequential Test for Two Arms: Conditional Biases from Upper and Lower Stopping Boundaries} \label{subSec::eg2}
Suppose we have two independent arms with unknown means $\mu_1$ and $\mu_2$. In this subsection, we consider the following testing problem:
\begin{equation}
H_0 : \mu_1 \leq \mu_2 ~~\text{vs}~~H_1: \mu_1 >\mu_2.
\end{equation}
To test this hypothesis, we draw a sample from arm $1$ for every odd time and from arm $2$ for every even time. Then, at each even time $t$, we check whether the difference between the sample means $\hat{\mu}_1(t)$, $\hat{\mu}_2(t)$ from the two arms crosses predefined upper and lower stopping boundaries, $t \mapsto U(t)$ and $t \mapsto  L(t)$. 

To be specific, define stopping times $\Tau^U$ and $\Tau^L$ as follows: 
\begin{align}
\Tau^U&:=\inf\left\{t \in \mathbb{N}_{\mathrm{even}} : \hat{\mu}_1(t) - \hat{\mu}_2(t) \geq U(t)  \right\}, \\
\Tau^L&:=\inf\left\{t \in \mathbb{N}_{\mathrm{even}} : \hat{\mu}_1(t) - \hat{\mu}_2(t) \leq L(t)  \right\}.
\end{align}
Let $M > 0$ be a predetermined maximum time budget. Based on $\Tau^U, \Tau^L$ and $M$, we stop sampling whenever $\hat{\mu}_1(t) - \hat{\mu}_2(t) $ crosses one of the boundaries or the maximum time budget $M$ is met. Define the corresponding stopping time as $\Tau_M := \min\left\{\Tau^U, \Tau^L, M\right\}$.
At time $\Tau_M$, we accept $H_1$ if $\Tau_M = \Tau^U$ (upper-crossing event), and accept $H_0$ if $\Tau_M = \Tau^L$  (lower-crossing event). Otherwise, we declare that we do not have enough evidence to accept either one of two hypotheses.  

% We can check this sequential process control false acceptation of $H_0$, $H_1$ at level $\alpha, \beta$ respectively. 

In this case, we cannot apply  \cref{prop::marginal_bias}  to determine the sign of the marginal bias since the stopping rule is neither optimistic nor pessimistic, marginally. However, we can determine the sign of the conditional bias based on \cref{thm::cond_bias} since the the function $\D_{\infty}^* \mapsto \frac{\mathbbm{1}(\Tau^U  \leq \min\left\{\Tau^L , M \right\})}{\Tau^U}$ is an increasing (resp. decreasing)  function of $X_{i,1}^*$ (resp. $X_{i,2}^*$) for each $i$, keeping all other entries in $\D_{\infty}^*$ fixed.

% the function $\D_{\infty}^* \mapsto \frac{\mathbbm{1}(t < \Tau)}{\Tau_M} = \frac{\mathbbm{1}(t < \Tau)}{t}$ is an decreasing function of $X_i^*$ 
In detail, conditioned on the event of accepting $H_0$, the sample mean and empirical
 CDF for arm $1$ are  negatively and positively biased, respectively. That is,
 
 \begin{align}
&\mathbb{E}\left[\hat{\mu}_1(\Tau_M) \mid \Tau_M = \Tau^L\right] \nonumber\\
&= \mathbb{E}\left[\hat{\mu}_1(\Tau^L) \mid \Tau^L \leq \min\left\{\Tau^U , M \right\}\right]  \leq \mu_1, \label{eq::eg2_naive_H0_mean}\\ 
&\inf_{y \in \mathbb{R}}\left( \E\left[\hat{F}_{k, \Tau_M} (y)\mid \Tau_M = \Tau^L  \right] - F_k(y) \right) \nonumber\\
&= \inf_{y \in \mathbb{R}}\left( \E\left[\hat{F}_{k, \Tau^L} (y)\mid \Tau^L \leq \min\left\{\Tau^U , M \right\} \right] - F_k(y) \right) \nonumber\\
& \geq 0.\label{eq::eg2_naive_H0}
\end{align}

Similarly, for arm 2, we have opposite signs of the sample mean and empirical CDF. 

By the same reasoning, the signs of the conditional biases conditioned on the event of accepting $H_1$ are reversed: %which has an opposite direction compared to the condition accepting $H_1$:

\begin{align}
&\mathbb{E}\left[\hat{\mu}_1(\Tau_M) \mid \Tau_M = \Tau^U\right] \nonumber\\
&= \mathbb{E}\left[\hat{\mu}_1(\Tau^U) \mid \Tau^U \leq \min\left\{\Tau^L , M \right\}\right]  \geq \mu_1, \label{eq::eg2_naive_H1_mean} \\ 
&\sup_{y \in \mathbb{R}}\left( \E\left[\hat{F}_{k, \Tau_M} (y)\mid \Tau_M = \Tau^U  \right] - F_k(y) \right) \nonumber\\
&= \sup_{y \in \mathbb{R}}\left( \E\left[\hat{F}_{k, \Tau^U} (y)\mid \Tau^U \leq \min\left\{\Tau^L , M \right\} \right] - F_k(y) \right) \nonumber\\
& \leq 0. \label{eq::eg2_naive_H1}
\end{align}

Thus, we conclude that, in all cases, the expected difference between the sample means are exaggerated toward ``the direction of decision''.

\begin{remark}
    These results hold regardless of whether the underlying distribution is under the null or alternative.
\end{remark}

% \begin{remark}
% Instead of imposing a fixed budget of $M$, we can also tweak this procedure to test the same hypothesis with $\epsilon$-indifference. We can conclude the same bias result for the $\epsilon$-indifference setting.
% \end{remark}

\begin{figure}
    \begin{center}
    \includegraphics[scale =  0.65]{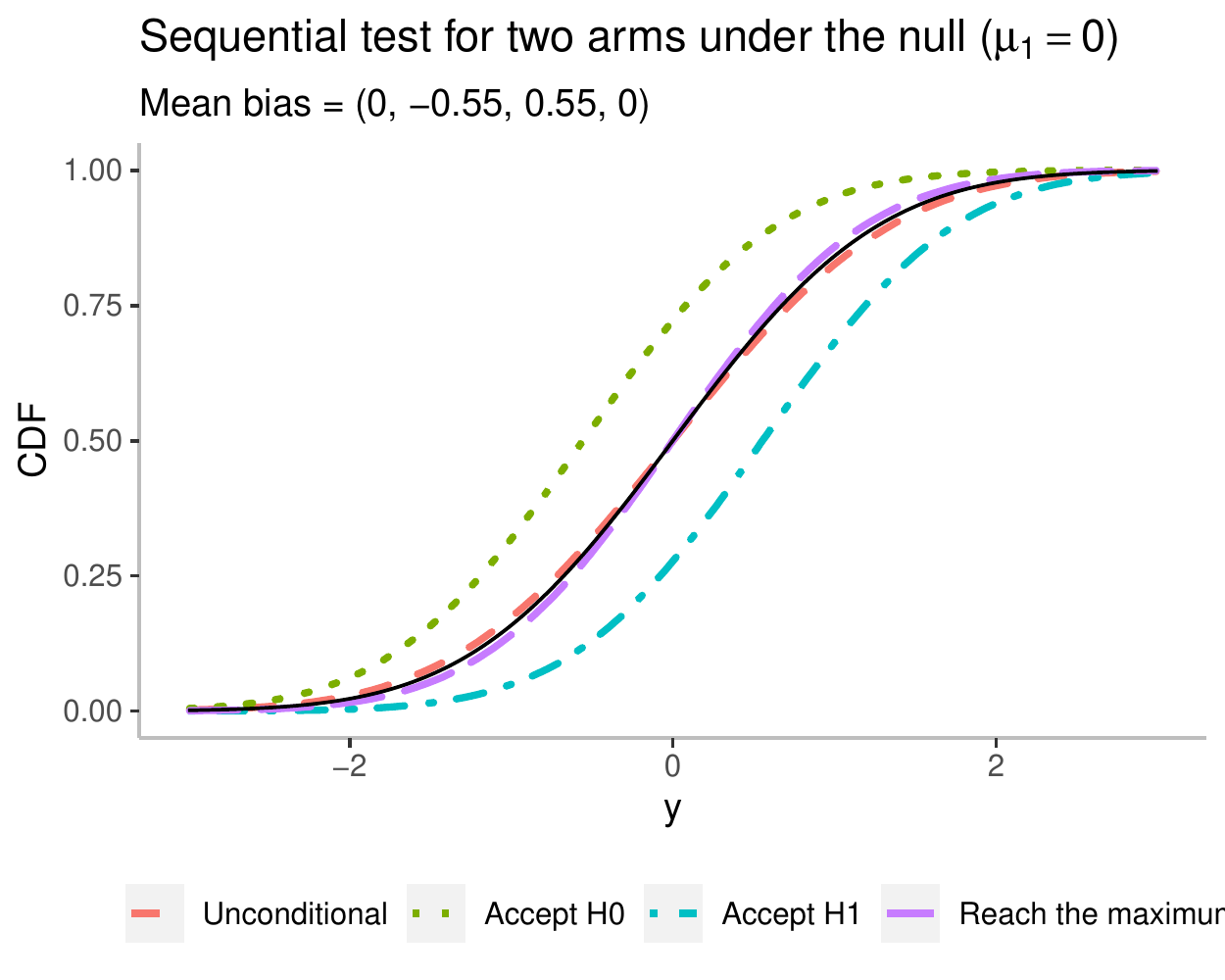}
    \end{center}
    \caption{\em  Average of marginal and conditional empirical CDFs of arm 1 from repeated sequential tests for two arms in \cref{subSec::eg2}. The black solid line refers to the true CDF of the underlying arm. }%Marginally and conditionally on the reaching the maximum time, the empirical CDFs are neither positive nor negatively biased across all $y \in \mathbb{R}$ (red dashed and purple long-dashed lines). On the other hand, conditioned on the accepting $H_0$ and $H_1$ events, the empirical CDFs are  positively and negatively biased, respectively as expected (green dotted and blue dot-dashed liens).}
    \label{fig::example2}
\end{figure}

{\bf Experiment:} To demonstrate the conditional bias result, we set two standard normal arms with same means $\mu_1 = \mu_2  = 0$. In this experiment, we use  upper and lower stopping boundaries based on naive point-wise confidence intervals:
\begin{equation}
    U(t) := z_{\alpha /2}\sqrt{\frac{2}{t}}, \quad \text{and} \quad L(t) = - U(t),
\end{equation}
where $\alpha$ is set to $0.2$ to show the bias better. \cref{fig::example2} show the marginal and conditional biases of the empirical CDFs and sample means for arm $1$ based on $10^5$ repetitions of the experiment. The black solid line corresponds to the true underlying CDF. The red dashed line refers to the average of the marginal CDFs, and the purple long-dashed line corresponds to the average of the empirical CDFs conditioned on reaching the maximal time. Note that for these two cases, the empirical CDFs are neither positively nor negatively biased across all $y \in \mathbb{R}$.

However, for the cases corresponding to accepting $H_0$ (green dotted line) and accepting $H_1$ (blue dot-dashed line), the empirical CDFs are positively and negatively biased, respectively: see inequalities~\eqref{eq::eg2_naive_H0}~and~ \eqref{eq::eg2_naive_H1}. The  conditional bias of the sample mean is also negative conditioning on the event of accepting $H_0$  and positive conditioning on the event of accepting  $H_1$: see inequalities~\eqref{eq::eg2_naive_H0_mean}~and~ \eqref{eq::eg2_naive_H1_mean}.

% Again, in all cases, the expected difference between sample means are exaggerated toward ``the direction of decision''.

\subsection{Best-Arm Identification: (i) lil’UCB Algorithm} \label{subSec::best_arm}
% \textcolor{red}{every gap is positive but Second best gap? Elimination? Top-k verses others? by using elimination?}

Suppose we have $K > 2$  sub-Gaussian arms with a common and known parameter $\sigma^2$. In many applications, we may want to identify which of the $K$ arms has the largest mean parameter by using as few samples as possible.

In the previous subsection, we observed that, in the two-armed bandit setting, if we use a boundary-crossing compatible with a best-arm identification style algorithm with deterministic sampling, then the optimal arm has positive bias and the sub-optimal arm has a negative bias.

\citet{shin2019bias} showed that the same phenomenon happens for the best arm as chosen  by the lil'UCB algorithm, which consists of the following sampling, stopping and choosing rules \citep{jamieson_lil_2014}:
\begin{itemize}
    \item Sampling: For $t =1,\dots K$, draw a sample from each arm. For $t > K$, draw a sample from arm $k$ if 
    \[
     k = \argmax_{j \in [K]} \hat{\mu}_j(t-1) + u^{\text{lil}}\left(N_j(t-1)\right),
    \]
    where $\beta, \epsilon, \lambda > 0$ and $\delta >0$ are algorithm parameters and  
    \begin{align*}
        u^{\text{lil}}\left(n\right):= &(1+\beta)(1+\sqrt{\epsilon})\\
        &\times \sqrt{2\sigma^2(1+\epsilon)\log\left(\log((1+\epsilon)n) /\delta\right)/n}.        
    \end{align*}
    \item Stopping: the stopping time $\Tau$ is defined as the first time at which the following inequality holds for some $k \in [K]$:
    \[
    N_k(t) \geq 1 + \lambda\sum_{j\neq k}N_j(t),
    \]
    where $\lambda>0$ is an algorithm parameter.
    \item Choosing: $\kappa = \argmax_{k \in [K]} N_k(\Tau)$.
\end{itemize}

\begin{figure}
    \begin{center}
    \includegraphics[scale =  0.65]{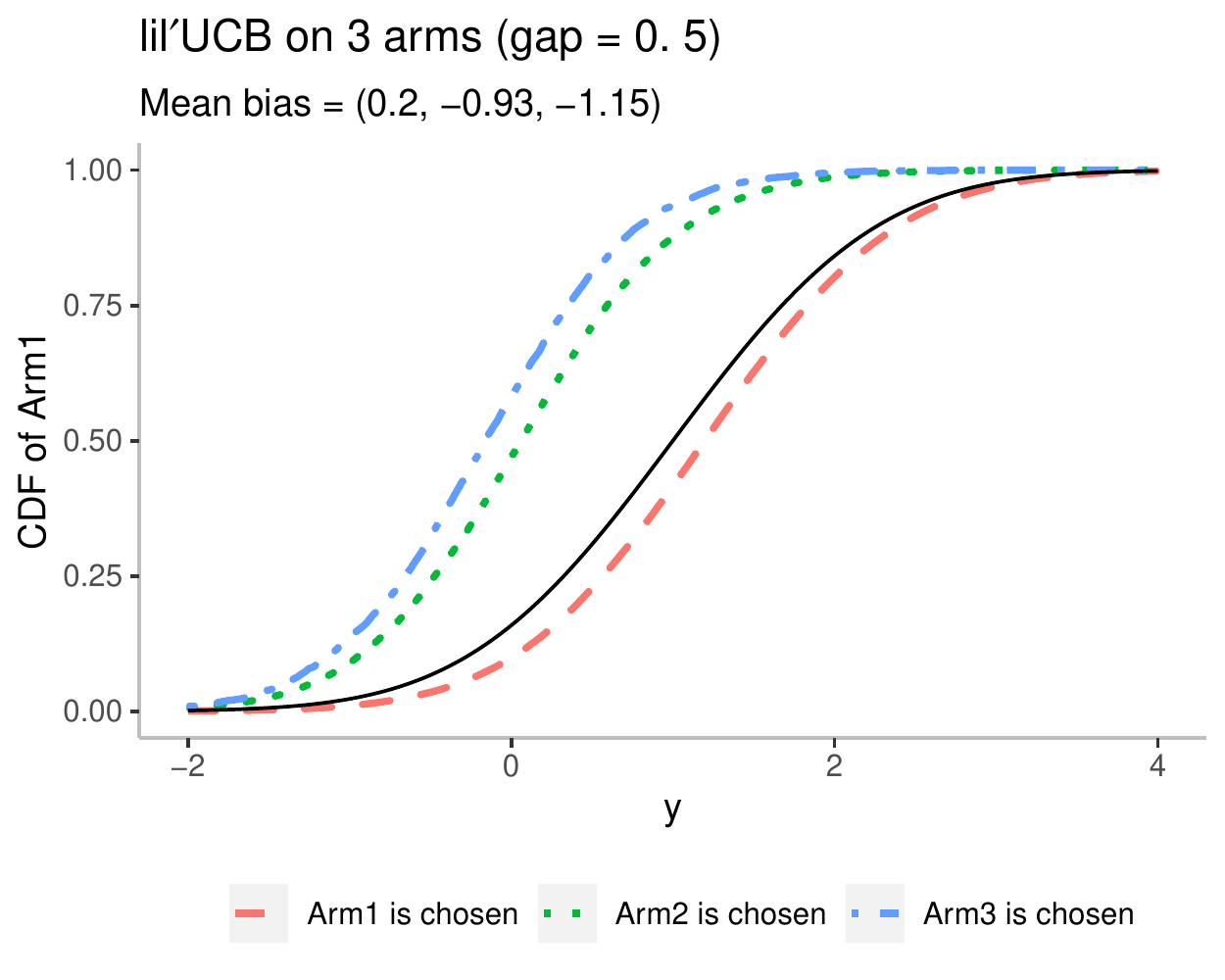}
    \end{center}
    \caption{\em  Average of conditional empirical CDFs of arm 1 from  $10^5$ lil'UCB algorithm runs on three unit-variance normal arms with $\mu_1= 1, \mu_2 = 0.5$ and $\mu_3 = 0$, as described in \cref{subSec::best_arm}. The black solid line refers to the true CDF of arm 1.}% Conditioned on the event arm 1 is chosen as the best arm ($\kappa =1$), the empirical CDF is negatively biased (red dashed line). On the other hand, conditioned on the event arm 1 is not chosen as the best arm ($\kappa \neq 1$), the empirical CDF is positively biased (green dotted and blue dot-dashed lines). }
    \label{fig::example4}
\end{figure} 
\begin{proposition}[Corollary~8 in \citet{shin2019bias}]\label{prop::lilUCB}
In the settings of the lil'UCB algorithm, for each $k$ with $\mathbb{P}(\kappa = k) >0$, we have that
\begin{equation}
    \mathbb{E}\left[ \hat{\mu}_k(\Tau) \mid \kappa = k \right]\geq \mu_k.
\end{equation}
\end{proposition}

The following corollary to \cref{thm::cond_bias} shows a complementary result---the sample mean and the empirical CDF of a given arm are negatively and positively biased, respectively, conditioned on the event that the arm is not selected as the best arm. In contrast, on the same conditioning event, the empirical distribution of the selected arm is negatively biased. 
\begin{corollary} \label{cor::lilUCB}
    In the settings of the lil'UCB algorithm, for each $k$ with $\mathbb{P}(\kappa \neq k) >0$, we have that 
\begin{align}
    \mathbb{E}\left[ \hat{\mu}_k(\Tau) \mid \kappa \neq k \right]&\leq \mu_k,  \\
\inf_{y \in \mathbb{R}}\left( \E\left[\hat{F}_{k, \Tau} (y)\mid \kappa \neq k \right] - F_k(y) \right)& \geq 0. 
\end{align}
Also, for each $k$ with $\mathbb{P}(\kappa = k) >0$, 
\begin{align}
    \mathbb{E}\left[ \hat{\mu}_k(\Tau) \mid \kappa = k \right]&\geq \mu_k, \\
\sup_{y \in \mathbb{R}}\left( \E\left[\hat{F}_{k, \Tau} (y)\mid \kappa = k \right] - F_k(y) \right)& \leq 0. 
\end{align}
\end{corollary}

The proof of \cref{cor::lilUCB} can be found in \cref{Appen::proof_of_lilUCB}.

\begin{remark}
    \cref{cor::lilUCB} remains true for any other choice of $u^{\text{lil}}$ if the function $n \mapsto u^{\text{lil}}(n)$ is decreasing. 
\end{remark}

\begin{figure}
    \begin{center}
    \includegraphics[scale =  0.65]{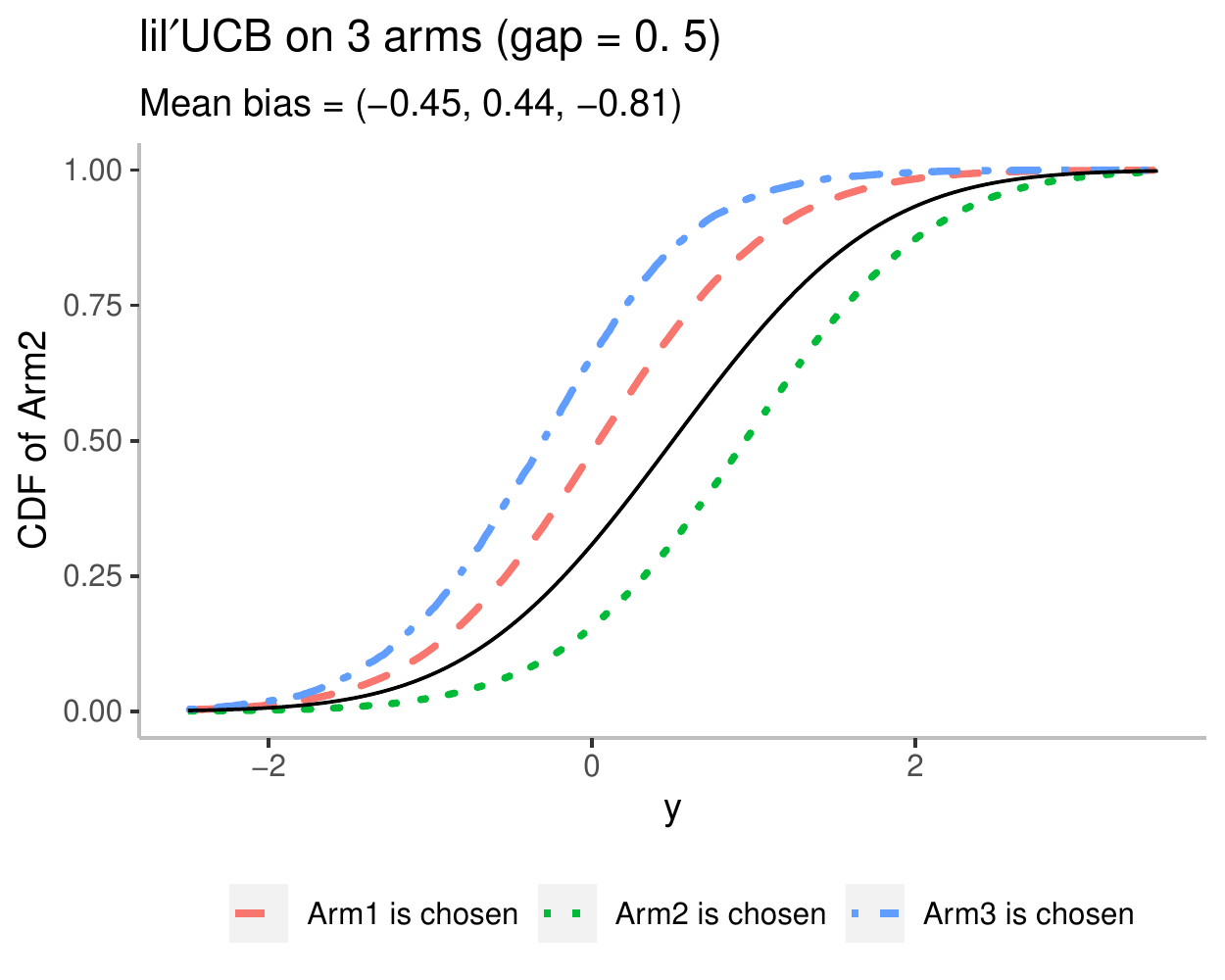}
    \end{center}
    \caption{\em   Average of conditional empirical CDFs of arm 2 from  $10^5$ lil'UCB algorithm runs on three unit-variance normal arms with $\mu_1= 1, \mu_2 = 0.5$ and $\mu_3 = 0$, as described in \cref{subSec::best_arm}. The black solid line refers to the true CDF of arm 2. }%Similar to the case of arm $1$, conditioned on the event arm 2 is chosen as the best arm ($\kappa =2$), the empirical CDF is negatively biased (green dotted line). On the other hand, conditioned on the event arm 2 is not chosen as the best arm ($\kappa \neq 2$), the empirical CDF is positively biased (red dashed and blue dot-dashed lines).}
    \label{fig::example5}
\end{figure}

{\bf Experiment:} To verify the previous claims, we conducted $10^5$ trials of the lil'UCB algorithm on three unit-variance normal arms with $\mu_1 = 1, \mu_2 = 0.5$ and $\mu_3 = 0$. It is important to note that the signs of the biases do not depend on the choice of parameters or of the underlying distributions, but the magnitudes of the biases do. To best illustrates the bias results, we use an unusual set of algorithm parameters as $\delta = 0.2$, $\epsilon = 0.1$, $\beta = 0.5$ and $\lambda = 1$ in this experiment.

\cref{fig::example4} shows the averages of the empirical CDFs of arm 1 (the arm with the larges mean) conditioned on each arm being chosen as the best arm. The black solid line corresponds to the true underlying CDF. The red dashed line, which lies below the true CDF, indicates that the empirical CDF of arm 1 conditioned on the event that the arm 1 is chosen as the best arm (i.e., $\kappa = 1$) is negatively biased; this then implies that the sample mean of the chosen arm is positively biased. In contrast, the green dotted and blue dot-dashed lines, lying above the true CDF, show that conditioned on the event the arm 1 is not chosen as the best arm (i.e., $\kappa \neq 1$), the empirical CDF is positively biased and the sample mean is negatively biased.

\cref{fig::example5} displays the averages of the empirical CDFs of arm 2. Though arm 2 is not the best arm, we can check that the signs of the conditional biases follow the same pattern as arm 1. Conditioned on the event that arm 2 is chosen as the best arm ($\kappa = 2$), the empirical CDF is negatively biased (green dotted line), but conditioned on the event that arm 2 is not chosen as the best arm ($\kappa \neq 2)$, the corresponding CDFs are now positively biased (red dashed and blue dot-dashed lines), as expected.
% ?from \cref{cor::lilUCB}.

% \section{Summary}\label{sec::summary}
% In this paper, we introduce a concise way to characterize the sign of the conditional bias in adaptive data collections and analyses, especially in the MAB framework. 
% Several interesting questions still remain open: (a) How to characterize the conditional bias for beyond linear and monotone functional such as variance and quantiles, (b) How to apply the knowledge of the sign of the bias to develop better debiased estimators, (c) Understanding the bias of regression functions in contextual bandit settings. 

\subsection{Best-Arm Identification: (ii) Sequential Halving} \label{subSec::SH}
% {\color{red}Make more general elimnation algorithm and explain many existing algorithms can be described in this way (not only for Top-K selection but general screening event. Explain by allowing no-elimination option, deterministic elimination time would be not too restrictive. }

The action elimination \citep[][]{paulson1964sequential, even2002pac, even2006action, karnin2013almost} 
is another class of best arm identification algorithms in which each arm in an active set is sampled a fixed number of times and some of arms are eliminated from the active set at each round. In this subsection, we focus on the sequential halving algorithm \citep{karnin2013almost}, which is designed to identify the best arm with high probability and based on a fixed number of samples. To be specific, the sequential halving algorithm for a given time budget $T$ is defined as follows :
 \begin{itemize}
    \item Initialize the set of active arms $\A_1 = [K]$.
    \item In each round $r = 1, \dots, \lceil \log_2 K\rceil $, for each arm $k$ in $\A_r$, draw  
    \[
    m_r:= \left\lfloor \frac{T}{|\A_r| \lceil \log_2 K\rceil} \right\rfloor
    \]
    i.i.d. samples from it, and let $\hat{\mu}_k^{(r)}$ be the average of the observed samples in the round. 
    \item Let $\mathcal{A}_{r+1}$ be the set of $\left\lceil |\A_r / 2| \right\rceil$ arms in $\A_r$ with the largest average based on data collected in round $r$.
    \item Choose the unique arm in $\A_{\lceil \log_2 K\rceil + 1}$ as the best one.
\end{itemize}

% Here, we assume the action elimination algorithm deploys a regular eliminating procedure in which if an arm $k$ with the sample mean $\hat{\mu}_k(t)$ were dropped out at time $t$, it should also be dropped at the same time with a smaller sample mean $\hat{\mu}_k^{'}(t)$ while all other observations remains same. Similarly, we assume if an arm $k$ with the sample mean $\hat{\mu}_k(t)$ remains active after time $t$, it should also be active with a larger sample mean $\hat{\mu}_k^{'}(t)$ while all other observations remains same. 

Let $\kappa$ be the (random) index of the chosen arm at the final time $T$, and 
for each $k \in [K]$, let $\hat{\mu}_k(T)$ be the sample mean of arm $k$ at time $T$ based on all previously collected $N_k(T)$ samples from that arm.

The following two corollaries show that under the sequential halving algorithm, the sample mean and empirical CDF of each arm is negatively and positively biased, respectively, marginally and conditionally on the event that the arm is not selected as the best one ($\kappa \neq k$). On the other hand, conditionally on the event that the arm is selected as the best one ($\kappa = k$), the sample mean and empirical CDF of each arm is positively and negatively biased, respectively. 

% [For the fixed at the stopping / Condition on t < Tau ]

\begin{corollary}\label{cor::SH_marginal}
Suppose that data are collected by the sequential halving algorithm with a time budget $T$. Then, for each $k \in [K]$ and for each $t \leq T$, we have that
\begin{align}
    \mathbb{E}\left[ \hat{\mu}_k(t)  \right]&\leq \mu_k, \\
\inf_{y \in \mathbb{R}}\left( \E\left[\hat{F}_{k, t} (y)\right] - F_k(y) \right)& \geq 0. 
\end{align}
\end{corollary}

\begin{corollary} \label{cor::SH_cond}
   In the settings of the sequential halving algorithm with a time budget $T$, for each $k$ with $\mathbb{P}(\kappa \neq k) >0$, we have that
\begin{align}
    \mathbb{E}\left[ \hat{\mu}_k(T) \mid \kappa \neq k \right]&\leq \mu_k, \\
\inf_{y \in \mathbb{R}}\left( \E\left[\hat{F}_{k, T} (y)\mid \kappa \neq k \right] - F_k(y) \right)& \geq 0. 
\end{align}
Also, for each $k$ with $\mathbb{P}(\kappa = k) >0$, 
\begin{align}
    \mathbb{E}\left[ \hat{\mu}_k(T) \mid \kappa = k \right]&\geq \mu_k, \\
\sup_{y \in \mathbb{R}}\left( \E\left[\hat{F}_{k, T} (y)\mid \kappa = k \right] - F_k(y) \right)& \leq 0. 
\end{align}
\end{corollary}

The proofs of \cref{cor::SH_marginal} and \ref{cor::SH_cond} can be found in \cref{Appen::proof_of_SH}. Due to limited space, an experiment about the bias of the sequential halving algorithm is presented in \cref{Appen::experiment_SH}.

\section{Summary}\label{sec::summary}
In this paper, we have investigated the sign of the conditional bias of monotone functions of the rewards in the MAB framework under an arbitrary conditioning event, generalizing the results in \citet{shin2019bias}, and complementing the results on the magnitude of the bias \citep{shin2019risk}. In our analysis, we have exploited certain natural monotonicity properties of  MAB experiments and have characterized the impact on the bias of both adaptivity in the data acquisition process and in the selection of the target for inference. %Our results offer novel insights and a better understanding of the bias in adaptive data analysis.

Several interesting extensions of our results are worth pursuing. We emphasize two important ones:
\begin{itemize}
    \item %Although the bias of the empirical CDF and its linear monotone functionals can describe the overall bias in data, 
    It is still an open problem how to characterize the bias (marginal or conditional) of other important functionals that are not necessarily monotone or linear, such as the sample variance and sample quantiles.
    \item Several debiasing methods have been proposed in the MAB literature: see, e.g., \citet{xu2013estimation,deshpande2017accurate, neel2018mitigating, nie2018adaptively,  hadad2019confidence}. However, the existing approaches typically only adjust for adaptive sampling but ignore the other sources of adaptivity. 
    \textbf{Current debiasing methods do not offer theoretical guarantees at stopping times, and also do not account for natural conditioning events in MABs, like selection of a promising arm.}  Furthermore, they are typically not interested in characterizing the \emph{sign} of the bias, and thus have complementary aims.  
    It is of interest to investigate how our results and techniques can be used in order to design more general debiased estimators.
    %\item A natural extension of our results would be to study the issue of bias in regression in the contextual bandits framework.
    %, we observe not only rewards but also obtain context variables related to the reward values. Given collected rewards and contexts, we can investigate their relationship through regression techniques. Understanding the bias of regression functions in contextual bandit algorithms is an important future direction.
    \end{itemize}


% \newpage
\bibliography{cond_bias}
\bibliographystyle{icml2020}

\appendix

\onecolumn
\clearpage
\newpage

\section*{\centerline{SUPPLEMENTARY MATERIAL}}
\section{Proofs}
\subsection{Proof of \cref{thm::cond_bias}}
\label{Appen::proof_of_cond_bias}
The proof of \cref{thm::cond_bias} is closely related to the proof of \cref{prop::marginal_bias} in \citet{shin2019bias}. However, in this proof, we decouple the bias part from the integrability condition which makes the proof significantly simpler than the one in \citet{shin2019bias}.

Under the condition in \cref{thm::cond_bias}, we first prove that if the function $\D^*_\infty \mapsto  \mathbbm{1}\left(C\right) / N_k(\Tau)$ is an decreasing function of $X_{i,k}^*$ while keeping all other entries in $\D^*_\infty$ fixed for each $i$ then, for any $t \in \mathbb{N}$ and $y \in \mathbb{R}$, the following inequality holds.
\begin{equation} \label{eq::main_lemma}
    \mathbb{E}\left[\frac{\mathbbm{1}(C)}{N_k(\Tau)}\mathbbm{1}(N_k(\Tau) < \infty) \mathbbm{1}\left(A_t = k\right) \left[\mathbbm{1}(Y_t \leq y) - F_k(y)\right]\right] \geq 0 
\end{equation}
\begin{proof}[Proof of inequality~\eqref{eq::main_lemma}]
First note that if $\D^*_\infty \mapsto  \mathbbm{1}\left(C\right) / N_k(\Tau)$ is a decreasing function of $X_{i,k}^*$ then the following function is also a decreasing function of $X_{i,k}^*$:
\begin{equation}
  D^*_\infty \mapsto  \frac{\mathbbm{1}\left(C\right)}{N_k(\Tau)}\mathbbm{1}(N_k(\Tau) < \infty) := h(\D^*_\infty)  
\end{equation}
Then, by the tabular representation of MAB, we can rewrite the LHS of inequality~\eqref{eq::main_lemma} as follows:
\begin{align*}
     &\mathbb{E}\left[\frac{\mathbbm{1}(C)}{N_k(\Tau)}\mathbbm{1}(N_k(\Tau) < \infty) \mathbbm{1}\left(A_t = k\right) \left[\mathbbm{1}(Y_t \leq y) - F_k(y)\right]\right] \\
     &=\mathbb{E}\left[h(\D_{\infty}^*)  \mathbbm{1}\left(A_t = k\right) \left[\mathbbm{1}(X_{N_k(t),k}^* \leq y) - F_k(y)\right]\right] \\
     &=\mathbb{E}\sum_{i=1}^t\left[h(\D_{\infty}^*) \mathbbm{1}\left(A_t = k, N_k(t) = i\right) \left[\mathbbm{1}(X_{i,k}^* \leq y) - F_k(y)\right]\right] \\
     &= \sum_{i=1}^t\mathbb{E}\left[h(\D_{\infty}^*) \mathbbm{1}\left(A_t = k, N_k(t) = i\right) \left[\mathbbm{1}(X_{i,k}^* \leq y) - F_k(y)\right]\right],
\end{align*}
where the third equality comes from the fact $N_k(t) \in \{1,\dots,t\}$. Therefore, to prove the inequality~\eqref{eq::main_lemma}, it is enough to show the following inequality:
\begin{equation}
    \mathbb{E}\left[ \mathbbm{1}\left(A_t = k, N_k(t) = i\right)h(\D_{\infty}^*) \left[\mathbbm{1}(X_{i,k}^* \leq y) - F_k(y)\right]\right] \geq 0.
\end{equation}
Note that the term $\mathbbm{1}\left(A_t = k, N_k(t) = i\right)$ does not depend on $X_{i,k}^*$ by the definition of $A_t$ and $N_k(t)$. 

Now, let $\D_{\infty}^{*'}$ be an other the tabular representation which is identical to $\D_{\infty}^{*}$ except the $(i,k)$-th entry of $X_\infty^*$ in $\D_{\infty}^*$ being replaced with an independent copy $X_{i,k}^{*'}$ from the same distribution $P_k$. 

Since the function $h$ is a decreasing function of $X_{i,k}^*$ while keeping all other entries in $\D^*_\infty$ fixed, we have that
\begin{equation}
    \left[h(\D_{\infty}^*)- h(\D_{\infty}^{*'})\right] \left[\mathbbm{1}(X_{i,k}^* \leq y) - \mathbbm{1}(X_{i,k}^{*'}\leq y)\right] \geq 0,
\end{equation}
which implies that
\begin{equation}
    \begin{aligned}
    &h(\D_{\infty}^*)\left[\mathbbm{1}(X_{i,k}^* \leq y) -F_k(y)\right] +     h(\D_{\infty}^{*'})\left[\mathbbm{1}(X_{i,k}^{*'} \leq y) -F_k(y)\right] +   \\
    & \geq h(\D_{\infty}^{*'})\left[\mathbbm{1}(X_{i,k}^* \leq y) -F_k(y)\right] +     h(\D_{\infty}^*)\left[\mathbbm{1}(X_{i,k}^{*'} \leq y) -F_k(y)\right].
    \end{aligned}
\end{equation}
By multiplying $\mathbbm{1}\left(A_t = k, N_k(t) = i\right)$ and taking expectations on both sides, we can show the inequality~\eqref{eq::main_lemma} hold as follows:
\begin{align}
    &2\mathbb{E}\left[\mathbbm{1}\left(A_t = k, N_k(t) = i\right)h(\D_{\infty}^*)\left[\mathbbm{1}(X_{i,k}^* \leq y) -F_k(y)\right]\right] \\
    &\geq 2\mathbb{E}\left[\mathbbm{1}\left(A_t = k, N_k(t) = i\right)h(\D_{\infty}^*)\left[\mathbbm{1}(X_{i,k}^{*'} \leq y) -F_k(y)\right]\right]\\
    &=  2\mathbb{E}\left[\mathbbm{1}\left(A_t = k, N_k(t) = i\right)h(\D_{\infty}^*)\right]\mathbb{E}\left[\mathbbm{1}(X_{i,k}^{*'} \leq y) -F_k(y)\right] \\
    & \geq 0,
\end{align}
where the first equality comes from the independence between $\mathbbm{1}\left(A_t = k, N_k(t) = i\right)h(\D_{\infty}^*)$ and $X_{i,k}^{*'}$, and the second inequality holds since $\mathbb{E}\left[\mathbbm{1}(X_{i,k}^{*'} \leq y)\right] = F_k(y)$. 
\end{proof}

Based on the inequality~\eqref{eq::main_lemma}, we are ready to prove \cref{thm::cond_bias}.
\begin{proof}[Proof of \cref{thm::cond_bias}]
% From the strong law of large numbers, for any fixed $y \in \mathbb{R}$, we have $ P_{N_k(\Tau)}\mathbbm{1}(N_k(\Tau) = \infty) = F_k(y)\mathbbm{1}(N_k(\Tau) = \infty)$ almost surely,\footnote{Technically, we define $\hat{F}_{k,\Tau}(y) := \lim_{t\to\infty}P_{N_k(\Tau \wedge t)}(y)$. Then, on the event $(N_k(\Tau) = \infty)$, the strong law of large numbers implies that $\hat{F}_{k,\Tau}(y) = F_k(y)$ almost surely on the event.}  which also implies that
% \begin{equation}\label{eq::expecation_on_infty}
%     \mathbb{E}\left[P_{N_k(\Tau)}\mathbbm{1}(N_k(\Tau) = \infty)\right] = F_k(y) \mathbb{P}\left(N_k(\Tau) = \infty\right).
% \end{equation}
% Therefore, the sign of $  \E\left[P_{N_{k}(\Tau)}(y) \mid C \right] - F_k (y)$

First, suppose the function $\D^*_\infty \mapsto  \mathbbm{1}\left(C\right) / N_k(\Tau)$ is an decreasing function of $X_{i,k}^*$ while keeping all other entries in $\D^*_\infty$ fixed for each $i$. Let $\{L_t\}_{t \in \mathbb{N}}$ be a sequence of random variables defined  as follows:
\begin{equation}
    L_t :=  \sum_{s=1}^t\frac{\mathbbm{1}(C)}{N_k(\Tau)}\mathbbm{1}(N_k(\Tau) < \infty) \mathbbm{1}\left(A_s = k\right) \left[\mathbbm{1}(Y_s \leq y) - F_k(y)\right],~~\forall t \in \mathbb{N}.
\end{equation}
From the inequality~\eqref{eq::main_lemma}, we have 
\begin{equation} \label{eq::main_lemma_implies}
\mathbb{E}\left[L_t\right] = \sum_{s=1}^t \mathbb{E}\left[\frac{\mathbbm{1}(C)}{N_k(\Tau)} \mathbbm{1}\left(A_s = k\right)\mathbbm{1}(N_k(\Tau) < \infty) \left[\mathbbm{1}(Y_s \leq y) - F_k(y)\right]\right] \geq 0,~~\forall t \in \mathbb{N}.
\end{equation}

Note that $N_k(\Tau) := \sum_{t=1}^\Tau \mathbbm{1}(A_t= k) = \sum_{t=1}^\infty \mathbbm{1}(A_t= k)$ since it is understood that for $t > \Tau$, $\mathbbm{1}(A_t = k) = 0$. Therefore, we know that, for each $y \in \mathbb{R}$, the sequence of random variables $\{L_t\}_{t\in\mathbb{N}}$ converges to $\left[\hat{F}_{k,\Tau}(y) - F_k(y)\right]\mathbbm{1}(C)\mathbbm{1}(N_k(\Tau) < \infty)$ almost surely. Also, it can be easily checked for each $t \in \mathbb{N}$, $|L_t|$ is upper bounded by $2$. Hence, from the dominated convergence theorem and the inequality~\eqref{eq::main_lemma_implies}, we have
\begin{align}
    0 \leq \lim_{t\to\infty}\mathbb{E}[L_t] &=\mathbb{E}\left[\left[\hat{F}_{k,\Tau}(y) - F_k(y)\right]\mathbbm{1}(C)\mathbbm{1}(N_k(\Tau) < \infty) \right] \\
    & =\mathbb{E}\left[\hat{F}_{k,\Tau}(y) \mathbbm{1}(C)\mathbbm{1}(N_k(\Tau) < \infty) \right] -  F_k(y)\mathbb{P}(C \cap \{N_k(\Tau) < \infty\}).
\end{align}
Since $\hat{F}_{k,\Tau}(y)\mathbbm{1}(N_k(\Tau) = \infty) = F_{k}(y)\mathbbm{1}(N_k(\Tau) = \infty)$ almost surely, the last inequality also implies that
\begin{equation}
    F_k(y)\mathbb{P}(C) \leq \mathbb{E}\left[\hat{F}_{k,\Tau}(y) \mathbbm{1}(C)\right]
\end{equation}
Since we assumed $\mathbb{P}(C) > 0$, by multiplying $1 / \mathbb{P}(C)$ on both sides, we have 
\begin{equation} \label{eq::main_ineq_in_proof}
    F_k(y) \leq \mathbb{E}\left[\hat{F}_{k,\Tau}(y) \mid C\right],    
\end{equation}
as desired. The inequality~\eqref{eq::main_ineq_in_proof} shows that the underlying distribution of arm $k$ stochastically dominates the empirical distribution of arm $k$ in the conditional expectation. In this case, it is well-known that for any non-decreasing integrable function $f$, the following inequality holds
\begin{equation}\label{eq::cond_bias_func}
    E_k f \geq \mathbb{E}\left[\hat{E}_{k,\Tau} f \mid C\right].
\end{equation}
For the completeness of the proof, we formally prove the inequality~\eqref{eq::cond_bias_func}. Since $f$ is integrable, without loss of generality, we may assume $f \geq 0$. For any $x \in \mathbb{R}$, define  $f^{-1}(x) := \inf\{y: f(y) > x\}$. Since $f$ is non-decreasing, for any probability measure $P$, the following equality holds 
\[
P\left(\left\{y: f(y) > x \right\}\right) = P\left(\left\{y: y > f^{-1}(x) \right\}\right), 
\]
for all but at most countably many $x\in \mathbb{R}$ which implies that
\begin{align}
     E_k f  &= \int_{0}^\infty P_k\left(\left\{y: f(y) > x \right\}\right)\mathrm{d}x \\
     &= \int_{0}^\infty 1-F_k\left(f^{-1}(x)\right)\mathrm{d}x \\
     &\geq \int_{0}^\infty 1-\mathbb{E}\left[\hat{F}_{k,\Tau}\left(f^{-1}(x)\right)\mid C\right]\mathrm{d}x \\
     & = \int_{0}^\infty \mathbb{E}\left[\hat{P}_{k,\Tau}\left(\left\{y: f(y) > x \right\}\right)\mid C\right]\mathrm{d}x \\
     & = \mathbb{E}\left[\hat{E}_{k,\Tau} f \mid C\right],
\end{align}
where the first and last equalities come from the Fubini's theorem with the integrability condition on $f$, and the first inequality comes from the inequality~\eqref{eq::main_ineq_in_proof}.

From the same argument with reversed inequalities, it can be shown that if the function $\D^*_\infty \mapsto  \mathbbm{1}\left(C\right) / N_k(\Tau)$ is an increasing function of $X_{i,k}^*$ while keeping all other entries in $\D^*_\infty$ fixed for each $i$, we have 
\begin{equation} \label{eq::main_ineq_in_proof_opposit}
    F_k(y) \geq \mathbb{E}\left[\hat{F}_{k,\Tau}(y) \mid C\right],~~\forall y \in \mathbb{R}.
\end{equation}
Equivalently, for any non-decreasing integrable function $f$, we have
\begin{equation}\label{eq::cond_bias_func_opposit}
    E_k f \leq \mathbb{E}\left[\hat{E}_k f \mid C\right],
\end{equation}
which completes the proof of Theorem~\ref{thm::cond_bias}.
\end{proof}

\subsection{Proof of \cref{cor::lilUCB}}
\label{Appen::proof_of_lilUCB}

Before proving \cref{cor::lilUCB} formally, we first provide an intuition as to why, for any reasonable and efficient algorithm for the best-arm identification problem, the sample mean and empirical CDF of an arm are negative and positive biases, respectively, conditionally on the event that  the arm is not chosen as the best arm. 

For any $k \in [K]$ and $i \in \mathbb{N}$, let $\D_\infty^*$ and $\D_\infty^{*'}$ be two collections of all possible arm rewards and external randomness that agree with each other except $X_{i,k}^* \geq X_{i,k}^{*'}$. Since we have a smaller reward from arm $k$ in the second scenario $\D_{\infty}^{*'}$, if $\kappa \neq k$ under the first scenario $\D_{\infty}^*$, any reasonable algorithm also would not pick the arm $k$ as the best arm under the more unfavorable scenario $\D_{\infty}^{*'}$. In this case, we know that $\kappa \neq k$ implies $\kappa' \neq k$. Also note that any efficient algorithm should be able to exploit the more unfavorable scenario $D_{\infty}^{*'}$ to easily identify arm $k$ as a suboptimal arm and choose another arm as the best one by using less samples from arm $k$. Therefore, we would have $N_k(\Tau) \geq N_k'(\Tau')$. As a result, we can expect that, from any reasonable and efficient algorithm, we would have $    \frac{\mathbbm{1}(\kappa \neq k)}{N_k(\Tau)} \leq  \frac{\mathbbm{1}(\kappa' \neq k)}{N_k'(\Tau')}$ which implies that for each $i$, the function $\D^*_\infty \mapsto  \mathbbm{1}\left(C\right) / N_k(\Tau)$ is a decreasing function of $X_{i,k}^*$ while keeping all other entries in $\D^*_\infty$ fixed. Then, from \cref{thm::cond_bias}, we have that the sample mean and empirical CDF of arm $k$ are negatively and positively biased conditionally on the event $\kappa \neq k$, respectively. Below, we formally verify that this intuition works for the lil'UCB algorithm. The proof is based on the following two facts about the lil'UCB algorithm:
\begin{itemize}
    \item {\bf Fact 1.} The lil'UCB algorithm has an optimistic sampling rule. That is,  for any fixed , $i, t \in \mathbb{N}$ and $k \in [K]$, the function  $\D^*_\infty \mapsto  N_k(t)$ is an increasing function of $X_{i,k}^*$  while keeping all other entries in $\D^*_\infty$ fixed \citep[see Fact 3 in][]{shin2019bias}.
    \item {\bf Fact 2.} Let $\D_\infty^*$ and $\D_\infty^{*'}$ be two collections of all possible arm rewards and external randomness that  agree with  each other  except in their $k$-th column of stacks of rewards $X_{\infty}^*$ and $X_{\infty}^{*'}$. For $j \in [K]$, let $N_j(t)$ and $N_j'(t)$ be the numbers of draws from arm $j$ under $\D_\infty^*$ and $\D_\infty^{*'}$ respectively. Then for each $t \in \mathbb{N}$, the following implications hold for lil'UCB algorithm  \citep[see Fact 3 and Lemma 9 in][]{shin2019bias}:
    \begin{align*} 
    N_k(t) \leq  N_k'(t) \Rightarrow N_j(t) \geq  N_j'(t),    \quad \text{ for all } j \neq k, \\
    N_k(t) \geq N_k'(t) \Rightarrow N_j(t) \leq N_j'(t),    \quad \text{ for all } j \neq k,
    \end{align*}
    which also implies that
    \begin{equation}
N_k(t) = N_k'(t) \Rightarrow N_j(t) = N_j'(t), \quad \text{ for all } j \neq k. \nonumber
\end{equation}
\end{itemize}

\begin{proof}[Proof of Corollary~\ref{cor::lilUCB}]
For any given $i \in \mathbb{N}$ and $k \in [K]$, let $\D_\infty^*$ and $\D_\infty^{*'}$ be two collections of all possible arm rewards and external randomness that  agree with  each other  except  $(i,k)$-th entries, $X_{i,k}^*$ and $X_{i,k}^{*'}$ of their stacks of rewards. Let $(N_k(t), N_k'(t))$ denote the numbers of draws from arm $k$ up to time $t$. Let $(\Tau, \Tau')$ be the stopping times and $(\kappa, \kappa')$ be choosing functions of the lil'UCB algorithm under $\D_\infty^*$ and $\D_\infty^{*'}$ respectively. 

Suppose $X_{i,k}^{*} \geq X_{i,k}^{*'}$. To prove the claimed bias result, it is enough to show that the function $\D^*_\infty \mapsto  \mathbbm{1}\left(C\right) / N_k(\Tau)$ is a decreasing function of $X_{i,k}^*$ while keeping all other entries in $\D^*_\infty$ fixed which corresponds to prove the following inequality holds:
\begin{equation} \label{eq::lil_ineq}
    \frac{\mathbbm{1}(\kappa \neq k)}{N_k(\Tau)} \leq  \frac{\mathbbm{1}(\kappa' \neq k)}{N_k'(\Tau')}.
\end{equation}
% From the equation $\mathbbm{1}(\kappa \neq k) = \sum_{j \neq k} \mathbbm{1}(\kappa = j)$ and positivity of $N_k(\Tau)$, we can check that if the following inequality holds for each $j \in [K] \setminus \{k\}$ then the inequality~\eqref{eq::lil_ineq} also holds:
% \begin{equation} \label{eq::lil_ineq_each_j}
%     \frac{\mathbbm{1}(\kappa = j)}{N_k(\Tau)} \leq  \frac{\mathbbm{1}(\kappa' = j)}{N_k'(\Tau')}.
% \end{equation}
Note that if $\kappa = k$ or $N_k(\Tau) = \infty$, the inequality~\eqref{eq::lil_ineq} holds trivially. Therefore, for the rest of the proof, we assume $\kappa \neq k$ and $N_k(\Tau)< \infty$. 

We will first prove the inequality $N_k(\Tau) \geq N_k'(\Tau')$ holds. From Fact 1 and the assumption $X_{i,k}^* \geq X_{i,k}^{*'}$, we have $N_k(t)~\geq~N_k'(t)$ for any fixed $t >0$. Then, by Fact 2, we also have $N_j(t)~\leq~N_j'(t)$ for any $j \neq k$. Since $\sum_{i \neq j} N_i(t) = t - N_j(t)$ for all $t$, we can rewrite the lil'UCB stopping rule as stopping whenever there exists $j \in [K]$ such that the inequality $N_j(t) \geq \frac{1 + \lambda t}{1 + \lambda}$ holds. Therefore, from the definition of the stopping rule with the fact  $N_j(t)~\leq~N_j'(t)$ for any $t \geq 1$ and $j \neq k$, at the stopping time $\Tau$, we have 
\begin{equation} \label{eq::ineq_1}
\frac{1 + \lambda \Tau}{1+ \lambda} \leq N_j(\Tau)  \leq N_j'(\Tau),     
\end{equation}
for some $j \neq k$ which also implies that the stopping condition is also satisfied for arm $j$  at time $\Tau$ under $\D_\infty^{*'}$ which implies that the stopping time under $\D_\infty^{*'}$ must be at most $\Tau$. Therefore we have $\Tau' \leq \Tau$. Now, since the inequality $N_k(t)~\geq~N_k'(t)$ holds for any $t \geq 1$, we have $N_k(\Tau)~\geq~N_k'(\Tau)$. Finally, since $t\mapsto N_k'(t)$ is a non-decreasing function, we can conclude $N_k(\Tau) \geq N_k'(\Tau) \geq N_k'(\Tau')$.

Since we proved $N_k(\Tau) \geq N_k'(\Tau')$, to complete the proof of \cref{cor::lilUCB}, it is enough to show that $\kappa \neq k$ implies $\kappa' \neq k$. We prove this statement by the 
proof by contradiction. Suppose $\kappa \neq k$ but $\kappa'  = k$. Then, there exists $j \neq k$ such that $\kappa = j$. By the definition of $\Tau$ and $\kappa$, we know that
\begin{equation} \label{eq::lilUCB_lemma1}
N_j(\Tau)  > N_k(\Tau).    
\end{equation}
Similarly, we can show that
\begin{equation} \label{eq::lilUCB_lemma2}
N_j'(\Tau')  < N_k'(\Tau').    
\end{equation}
It is important to note that these inequalities are strict. Note that since we draw a singe sample at each time, if $N_j(\Tau) = N_k(\Tau)$ then at the time $\Tau - 1$, either arm $j$ or $k$ should satisfy the stopping rule which contradicts to the definition of $\Tau$.

Recall that, in \cref{eq::ineq_1}, we showed that if $\kappa = j$, at stopping time $\Tau$, we have
\[
   \frac{1 + \lambda \Tau}{1+ \lambda} \leq N_j(\Tau)  \leq N_j'(\Tau),
\]
which implies $\Tau' \leq \Tau$. By the same argument, at the stopping time $\Tau'$ with the assumption $\kappa' = k$, we have
\[
   \frac{1 + \lambda \Tau'}{1+ \lambda} \leq N_k'(\Tau')  \leq N_k(\Tau'),
\]
which also implies $\Tau \leq \Tau'$. From these two inequalities on stopping times, we have $\Tau' = \Tau$. Finally, by combining inequalities between pairs of $N_k, N_k', N_j, N_j'$ with the observation $\Tau' = \Tau$, we have
\[
 N_j'(\Tau') < N_k'(\Tau') \leq N_k(\Tau') = N_k(\Tau) < N_j(\Tau) \leq N_j'(\Tau) = N_j'(\Tau')
\]
where the first inequality comes from the inequality~\eqref{eq::lilUCB_lemma2}. The second inequality come from $N_k' \leq N_k$. The first equality comes from $\Tau' = \Tau$ and the third inequality comes from the inequality~\eqref{eq::lilUCB_lemma1}. The last inequality comes from $N_j \leq N_j'$ and the final equality comes from $\Tau = \Tau'$.

This is a contradiction, and, therefore, $\kappa \neq k$  implies that $\kappa' \neq k$. This proves that for each $i$, the function $\D^*_\infty \mapsto  \mathbbm{1}\left(C\right) / N_k(\Tau)$ is a decreasing function of $X_{i,k}^*$ while keeping all other entries in $\D^*_\infty$ fixed and, from \cref{thm::cond_bias}, we can conclude that the sample mean and empirical CDF of arm $k$ from the lil'UCB algorithm are negatively and positive biased conditionally on the event the arm $k$ is not chosen as the best arm, respectively.
\end{proof}

\subsection{Proofs of \cref{cor::SH_marginal} and \ref{cor::SH_cond}} \label{Appen::proof_of_SH}

For any given $i \in \mathbb{N}$ and $k \in [K]$, let $\D_\infty^*$ and $\D_\infty^{*'}$ be two infinite collections of arm rewards and external randomness that  agree with  each other  except  on the $(i,k)$-th entries, $X_{i,k}^*$ and $X_{i,k}^{*'}$ of their stacks of rewards. Without loss of generality, we assume $X_{i,k}^* \leq X_{i,k}^{*'}$. Finally, let $(N_k(t), N_k'(t))$ denote the numbers of draws from arm $k$ up to time $t \leq T$.

We first prove that the following inequality holds:
\begin{equation} \label{eq::SH_basic_ineq}
N_k(t) \leq N_k'(t),~~\forall t \leq T.
\end{equation}
From \cref{thm::cond_bias}, the above inequality implies that the sample mean and empirical CDF of each arm at a fixed time are negatively and positively biased, respectively as claimed in \cref{cor::SH_marginal}.

\begin{proof}[Proof of inequality~\eqref{eq::SH_basic_ineq}]
For each $t \leq T$, let $R_k(t)$ and $R_k'(t)$ be numbers of rounds in which arm $k$ is in the active set up to time $t$ under $\D_\infty^*$ and $\D_\infty^{*'}$, respectively. Since $N_k(t) = \sum_{r=1}^{R_k(t)}m_r$ and $N_k'(t) = \sum_{r=1}^{R_k'(t)}m_r$, it is enough to show $R_k(t) \leq R_k'(t)$. 
Now, for the sake of contradiction, suppose $R_k(t) > R_k'(t)$. Since arm $k$ is active at rounds $r = 1, \dots, R_k'(t)$ under both  $\D_\infty^*$ and $\D_\infty^{*'}$, we must have that $\hat{\mu}_k^{(r)} \leq \hat{\mu}_k^{'(r)}$ for each $r = 1, \dots, R_k'(t)$, where $\hat{\mu}_k^{(r)}$ and $\hat{\mu}_k^{'(r)}$ are averages of the observed samples from arm $k$ at each round $r$ under $\D_\infty^*$ and $\D_\infty^{*'}$, respectively. 
By the same logic, we have that $\A_r = \A_r'$ for each $r = 1, \dots, R_k'(t)$, where $\A_r$ and $\A_r'$ are the sets of active arms under $\D_\infty^*$ and $\D_\infty^{*'}$, respectively. In particular, at the $R_k'(t)$-th round, under both scenarios $\D_\infty^*$ and $\D_\infty^{*'}$,  we have the exactly same active sets and the same sample averages of all the active arms except for arm $k$, for which it hols that $\hat{\mu}_k^{(R_k'(t))} \leq \hat{\mu}_k^{'(R_k'(t))}$. However, it cannot be the case that arm $k$ is eliminated from the active set after the $R_k'(t)$-th round under $\D_\infty^{*'}$ but the same arm still remains active after the same round under the scenario $\D_\infty^{*}$ in which the arm $k$ has even a smaller sample average.  This proves  $R_k(t) \leq R_k'(t)$, which  implies the claimed inequality~\eqref{eq::SH_basic_ineq} as desired. 
%Since arm $k$ is active at least for the first $R_k'(t)$ rounds, and since $\D_\infty^*$ and $\D_\infty^{*'}$ are agreed to each other all but $X_{i,k}^*$ and $X_{i,k}^{*'}$, we have 
%However, it contradicts the fact that after $R_k'(t)$-th round, arm $k$ is eliminated from the active set under $\D_\infty^{*'}$ but the same arm still remain active after the same round under $\D_\infty^{*}$ in which the arm $k$ has even a smaller sample average.
\end{proof}

For the proof of \cref{cor::SH_cond}, let $\kappa$ and $\kappa'$ be indices of the chosen arm under $\D_\infty^*$ and $\D_\infty^{*'}$, respectively. To prove the first two inequalities, we need to to show that
\begin{equation}
    \frac{\mathbbm{1}(\kappa \neq k)}{N_k(T)} \geq  \frac{\mathbbm{1}(\kappa' \neq k)}{N_k'(T)}.
\end{equation}
From \cref{eq::SH_basic_ineq}, we have that $N_k(T) \leq N_k'(T)$. Therefore, it is enough to show that $\mathbbm{1}(\kappa \neq k) \geq \mathbbm{1}(\kappa' \neq k)$ which is equivalent to showing that
\begin{equation} \label{eq::SH_indicators}
    \mathbbm{1}(\kappa = k) \leq \mathbbm{1}(\kappa' = k).
\end{equation}
For the sake of contradiction, suppose $\kappa = k$ but $\kappa' \neq k$. Recall that $R_k(T)$ and $R_k'(T)$ are the number of rounds in which arm $k$ is in the active set up to time $T$ under $\D_\infty^*$ and $\D_\infty^{*'}$ respectively. Since  $\kappa = k$ but $\kappa' \neq k$, we have  $\lceil \log_2 K \rceil = R_k(T) > R_k'(T)$. However this contradicts the fact $N_k(T) \leq N_k'(T)$, since 
$N_k(T) = \sum_{r=1}^{R_k(T)}m_r$ and $N_k'(T) = \sum_{r=1}^{R_k'(T)}m_r$. Hence, inequality~\eqref{eq::SH_indicators} is true. By \cref{thm::cond_bias} this implies  the first two inequalities in \cref{cor::SH_cond}. 

Similarly, to prove the last two inequalities in \cref{cor::SH_cond}, we need to to show that
\begin{equation} \label{eq::SH_ineq2}
    \frac{\mathbbm{1}(\kappa = k)}{N_k(T)} \leq  \frac{\mathbbm{1}(\kappa' = k)}{N_k'(T)}.
\end{equation}
If $\kappa \neq k$, the above inequality holds trivially. Hence, we can assume $\kappa = k$ and, from the inequality~\eqref{eq::SH_indicators}, we also have $\kappa' = k$. Since arm $k$ is chosen as the best one under both $\D_\infty^*$ and $\D_\infty^{*'}$, we have that $N_k(T) = N_k'(T) = \sum_{r=1}^{\lceil \log_2 K \rceil}m_r$ which proves inequality~\eqref{eq::SH_ineq2}. The result again follows from  \cref{thm::cond_bias}.

% \subsection{integrability condition}
% \begin{remark}
% If a function $f$ is integrable with respect to $E_k$ then  the function $f$ also satisfies $\E\left[E_k |f(X)| \mid C \right] < \infty$ if one of the following conditions holds:
% \begin{enumerate}
%     \item $\mathbb{E}N_k(\Tau)$ is finite.
%     \item For any $p > 2$, $|f|^p$ is integrable with respect to $E_k$.
% \end{enumerate}
% It is an open question whether the integrability of $|f|$ with respect to $E_k$ for some $p \in [1,2]$ also implies $\E\left[E_k |f(X)| \mid C \right] < \infty$.
% \end{remark}

\section{Additional Simulation Results}
In this section, we present additional simulation results for Section~\ref{sec::simulations}~and~\ref{sec::summary} which are omitted from the main part due to the page limit.

\subsection{Conditional Bias Under Alternative Hypothesis in \cref{subSec::eg2}}

\begin{figure}
    \begin{center}
    \includegraphics[scale =  0.65]{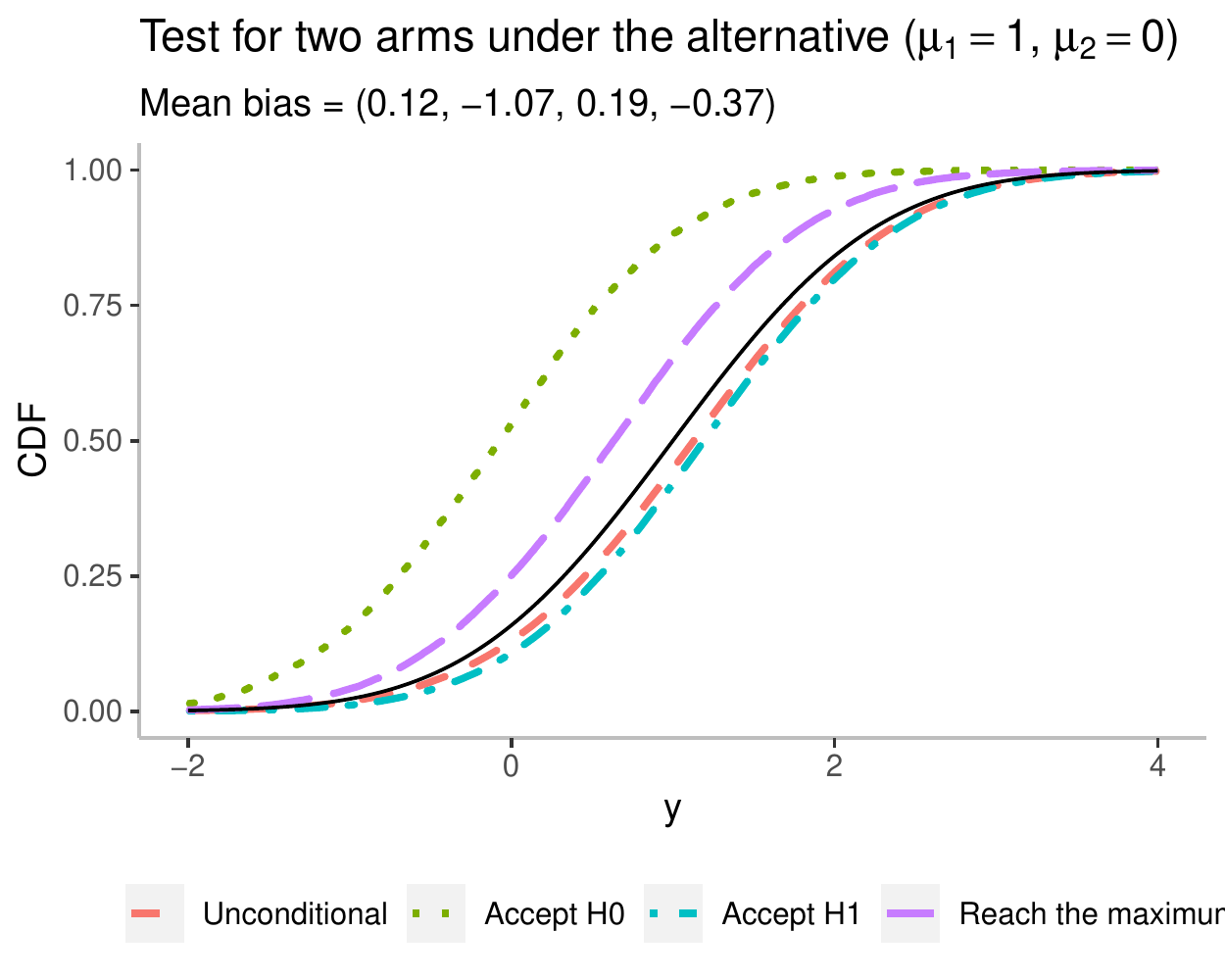}
    \end{center}
    \caption{\em  Average of conditional empirical CDFs of arm 1 from repeated sequential tests for two arms under the alternative hypothesis ($\mu_1 = 1, \mu_2 = 0$). See \cref{subSec::eg2} for the detailed explanation about the sequential test.} %The dashed line refers to the true CDF of the underlying arm. Marginally and conditionally on the reaching the maximum time, the empirical CDFs are neither positive nor negatively biased across all $y \in \mathbb{R}$. On the other hand, conditioned on the accepting $H_1$ and $H_0$ events, the empirical CDFs are negatively and positively biased, respectively as expected from  \cref{thm::cond_bias}.}
    \label{fig::example2_appen}
\end{figure}

As we conducted in \cref{subSec::eg2}, we have two standard normal arms with means $\mu_1$ and  $\mu_2$. Then, we use the following upper and lower stopping boundaries to test whether $\mu_1 \leq \mu_2$ or not:
\begin{equation}
    U(t) := z_{\alpha /2}\sqrt{\frac{2}{t}}, \quad \text{and} \quad L(t) = - U(t),
\end{equation}
where $\alpha$ is set to $0.2$ to show the bias better. In contrast to the experiment in \cref{subSec::eg2} in which the true means are equal to each other, in this experiment, we set $\mu_1 =1$ and $\mu_2 =0$ to make the alternative hypothesis is true.

\cref{fig::example2_appen} show the marginal and conditional biases of the empirical CDFs and sample means for arm $1$ based on $10^5$ repetitions of the experiment. The black solid line corresponds to the true underlying CDF. The red dashed line refers to the average of the marginal CDFs, and the purple long-dashed line corresponds to the average of the empirical CDFs conditionally on reaching the maximal time. For these two cases, although the marginal CDF is negatively and the conditional CDF is positively biased, these are not general phenomena and the sign of bias can be changed as we change mean parameters.

However, for the cases corresponding to accepting $H_1$ (blue dot-dashed line) and accepting $H_0$ (green dotted line), we can check that signs of biases of CDFs and sample means are consistent with what \cref{thm::cond_bias} and corresponding  inequalities \eqref{eq::eg2_naive_H1_mean} to \eqref{eq::eg2_naive_H0} described. Also note that the bias results do not depend on whether the arms are under the null or alternative hypotheses.

\subsection{Experiment About the Bias of the Sequential Halving Algorithm in \cref{subSec::SH}}
\label{Appen::experiment_SH}

\begin{figure}
    \begin{center}
    \includegraphics[scale =  0.65]{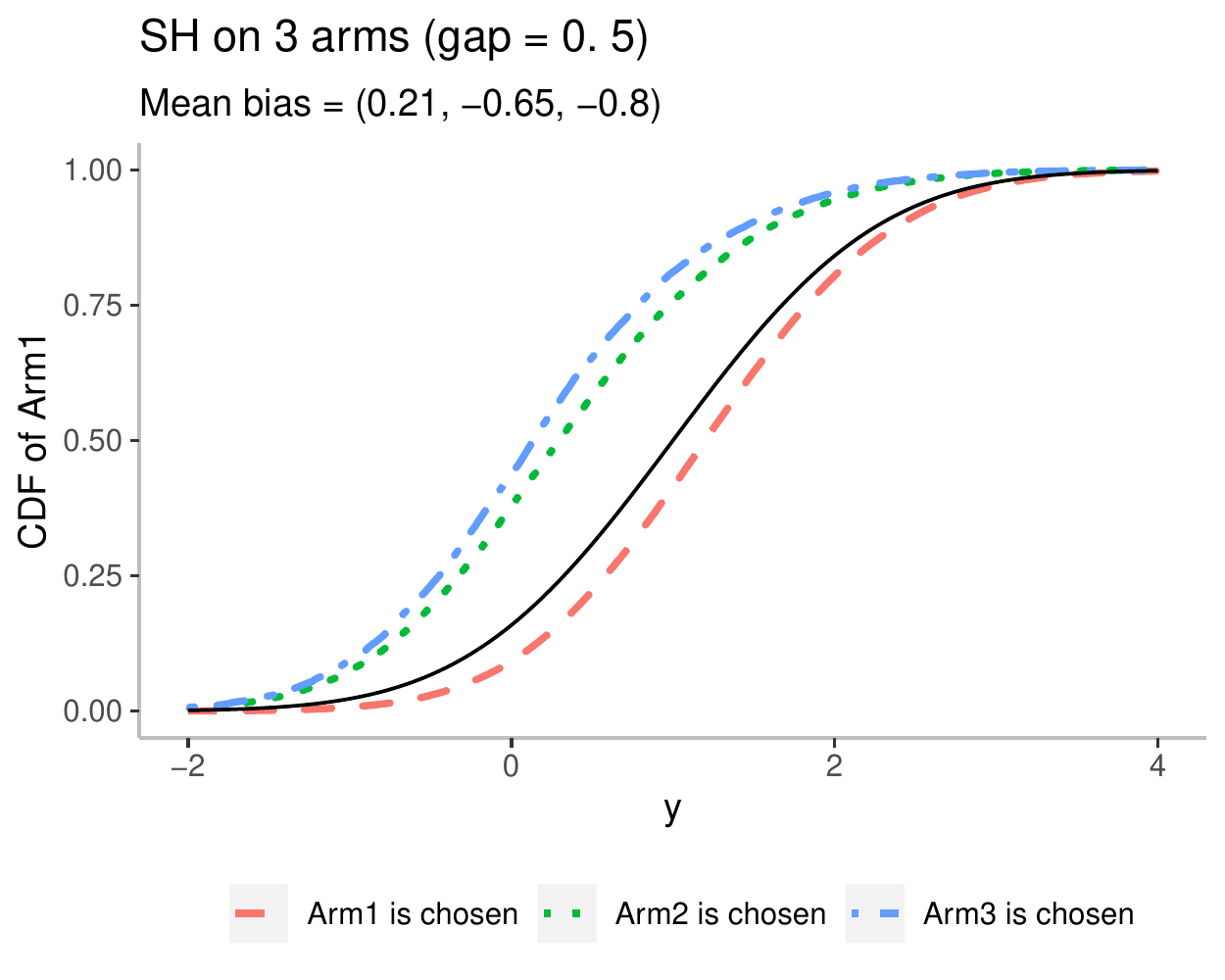}%
    \includegraphics[scale =  0.65]{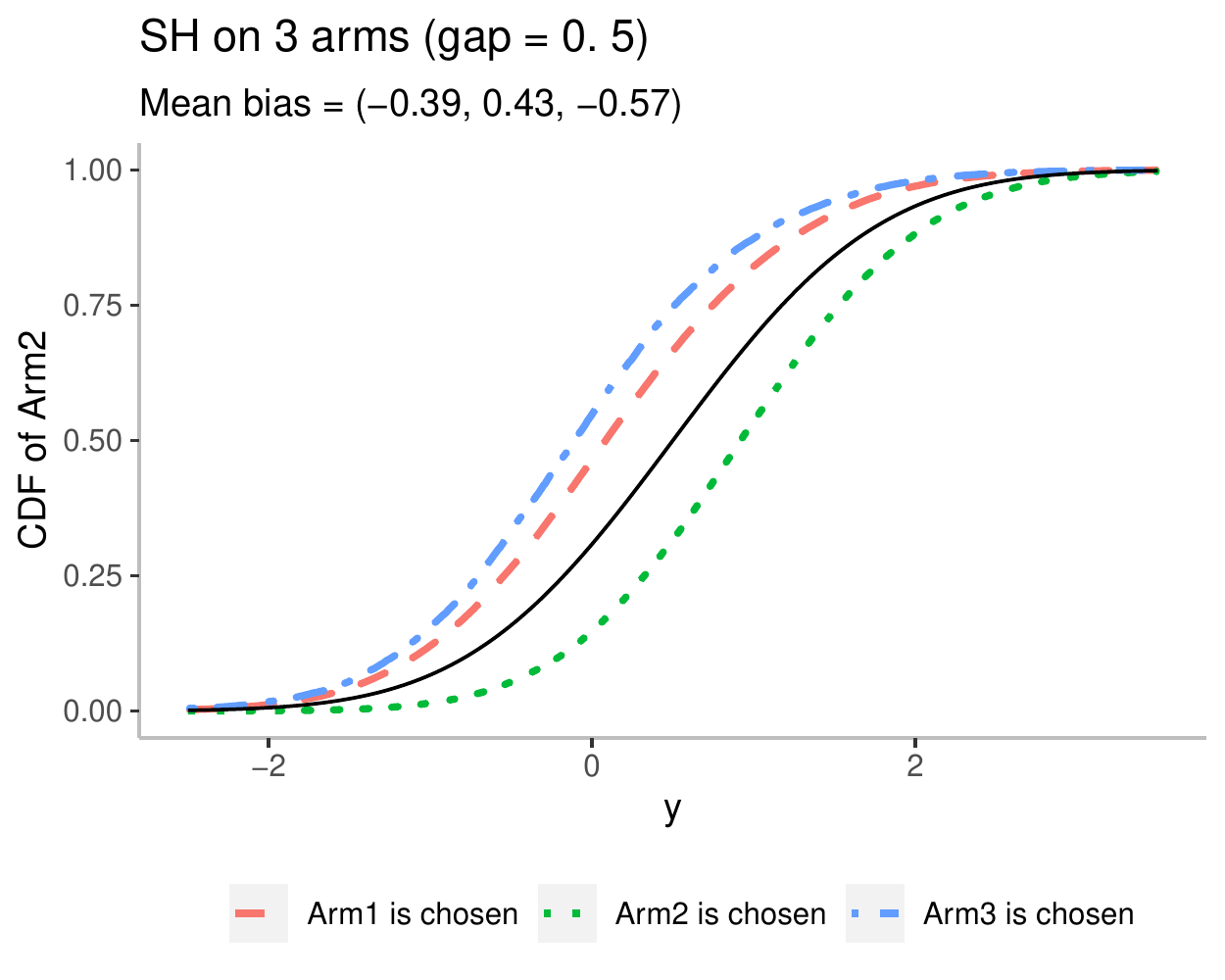}
    \end{center}
    \caption{\em Average of conditional empirical CDFs of arm 1 (left) and arm 2 (right) from  $10^5$ the sequential halving algorithm runs on three unit-variance normal arms with $\mu_1= 1, \mu_2 = 0.5$ and $\mu_3 = 0$, as described in \cref{Appen::experiment_SH}. Black solid lines refer to the true CDF of arm 1 and arm 2.}
    \label{fig::example1_6}
\end{figure}

 To verify the conditional bias results in \cref{cor::SH_marginal} and \ref{cor::SH_cond}, we conducted $10^5$ trials of the sequential halving algorithm on three unit-variance normal arms with $\mu_1 = 1, \mu_2 = 0.5$ and $\mu_3 = 0$ as we did for the lil'UCB algorithm in \cref{subSec::best_arm}. Again, it is important to note that the signs of the biases do not depend on the choice of parameters or of the underlying distributions, but the magnitudes of the biases do. To best illustrates the bias results, we use an unusually small time budget $T = 10$ in this experiment.

The left side of \cref{fig::example1_6} shows the averages of the empirical CDFs of arm 1 (the arm with the largest mean) conditionally on each arm being chosen as the best arm. The thick black line corresponds to the true underlying CDF. The red dashed line, which lies below the true CDF, indicates that the empirical CDF of arm 1 conditionally on the event that  arm 1 is chosen as the best arm (i.e., $\kappa = 1$) is negatively biased; this then implies that the sample mean of the chosen arm is positively biased. In contrast, the green dotted and blue dot-dashed lines, lying above the true CDF, show that conditionally on the event that arm 1 is not chosen as the best arm (i.e., $\kappa \neq 1$), the empirical CDF is positively biased and the sample mean is negatively biased.

The right side of \cref{fig::example1_6} displays the averages of the empirical CDFs of arm 2. Though arm 2 is not the best arm, we can check that the signs of the conditional biases follow the same pattern as arm 1. Conditionally on the event that arm 2 is chosen as the best arm ($\kappa = 2$), the empirical CDF is negatively biased (green dotted line), but conditionally on the event that arm 2 is not chosen as the best arm ($\kappa \neq 2)$, the corresponding CDFs are now positively biased (red dashed and blue dot-dashed lines), as expected.

\subsection{Experiments on Conditional Biases of Sample Variance and Median in MABs}
As stated in \cref{sec::summary}, characterizing the bias of other important functionals such as sample variance and sample quantiles is an important open problem. In this subsection, we present a simulation study on the bias of sample variance and median. 

For a given $n \geq 2$ i.i.d. samples $X_1,\dots, X_n$ from a distribution $P$, the sample variance $\hat{\sigma}^2$ and median $\hat{m}$ are defined by
\begin{align}
    \hat{\sigma}^2 &= \frac{1}{n-1} \sum_{i=1}^n (X_i - \bar{X}_n)^2 \\
    \hat{m} &=\begin{cases}
    \frac{1}{2}\left(X_{(n/2)} + X_{(n/2 + 1)}\right) &\mbox{if $n$ is even}, \\
     X_{(\frac{n+1}{2})}& \mbox{if $n$ is odd},
    \end{cases}
\end{align}
where $\bar{X}_n$ corresponds to the sample mean and $X_{(i)}$ refers to the $i$-th smallest sample. 

It is well-known that for any distribution $P$ with finite variance $\sigma^2$, the sample variance $\hat{\sigma}^2$ is an unbiased estimator of $\sigma^2$. Also,  though the sample median is not necessarily unbiased, for any symmetric distribution $P$ including the normal distribution as a special case, the sample median is unbiased. However, for adaptively collected data from a MAB experiment, it is unclear whether the sample variance and median are unbiased or not. Furthermore, it is an open question how to characterize the bias of sample variance and median estimators if they are biased estimators. 

\begin{figure}
    \begin{center}
    \includegraphics[scale =  0.65]{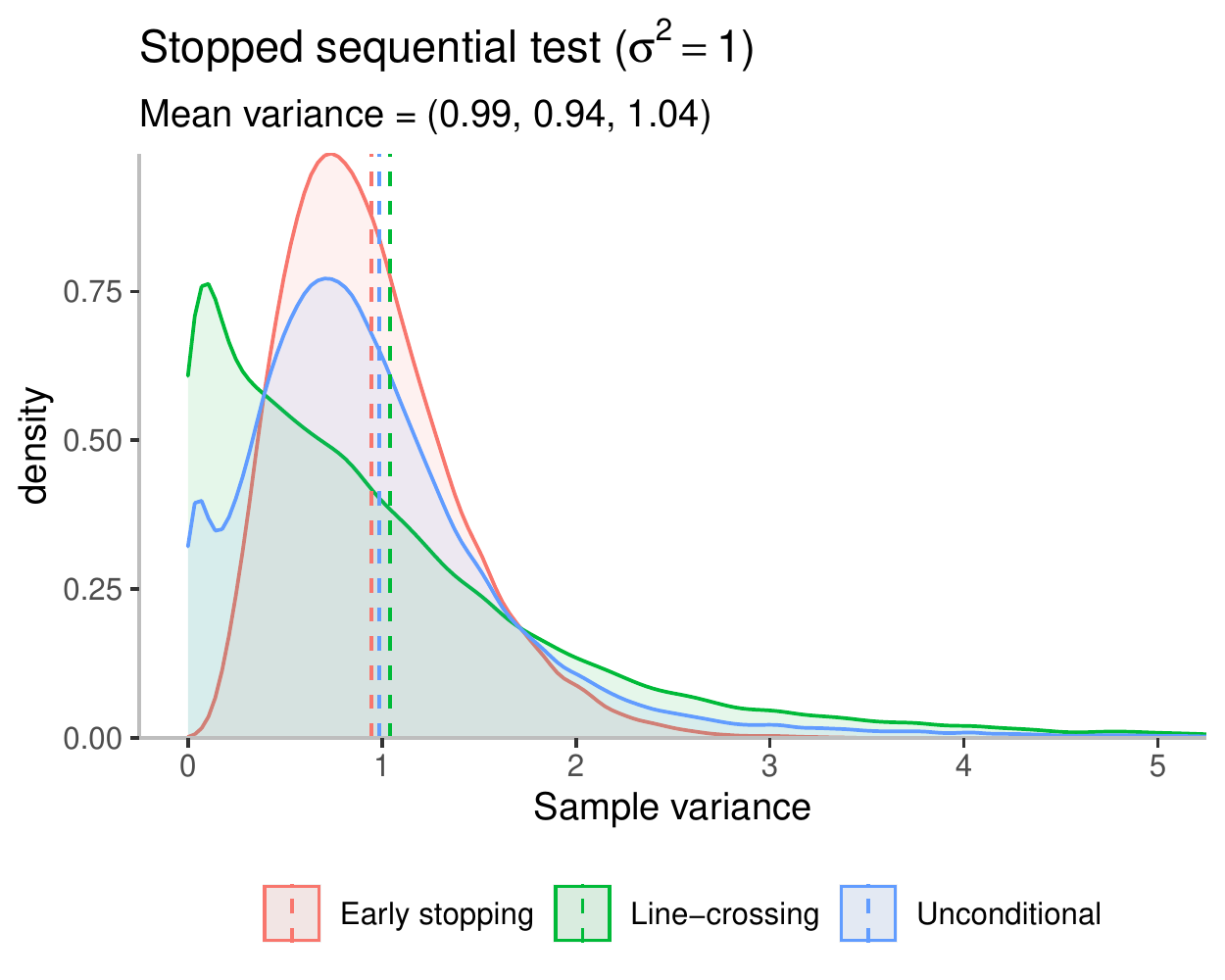}%
    \includegraphics[scale =  0.65]{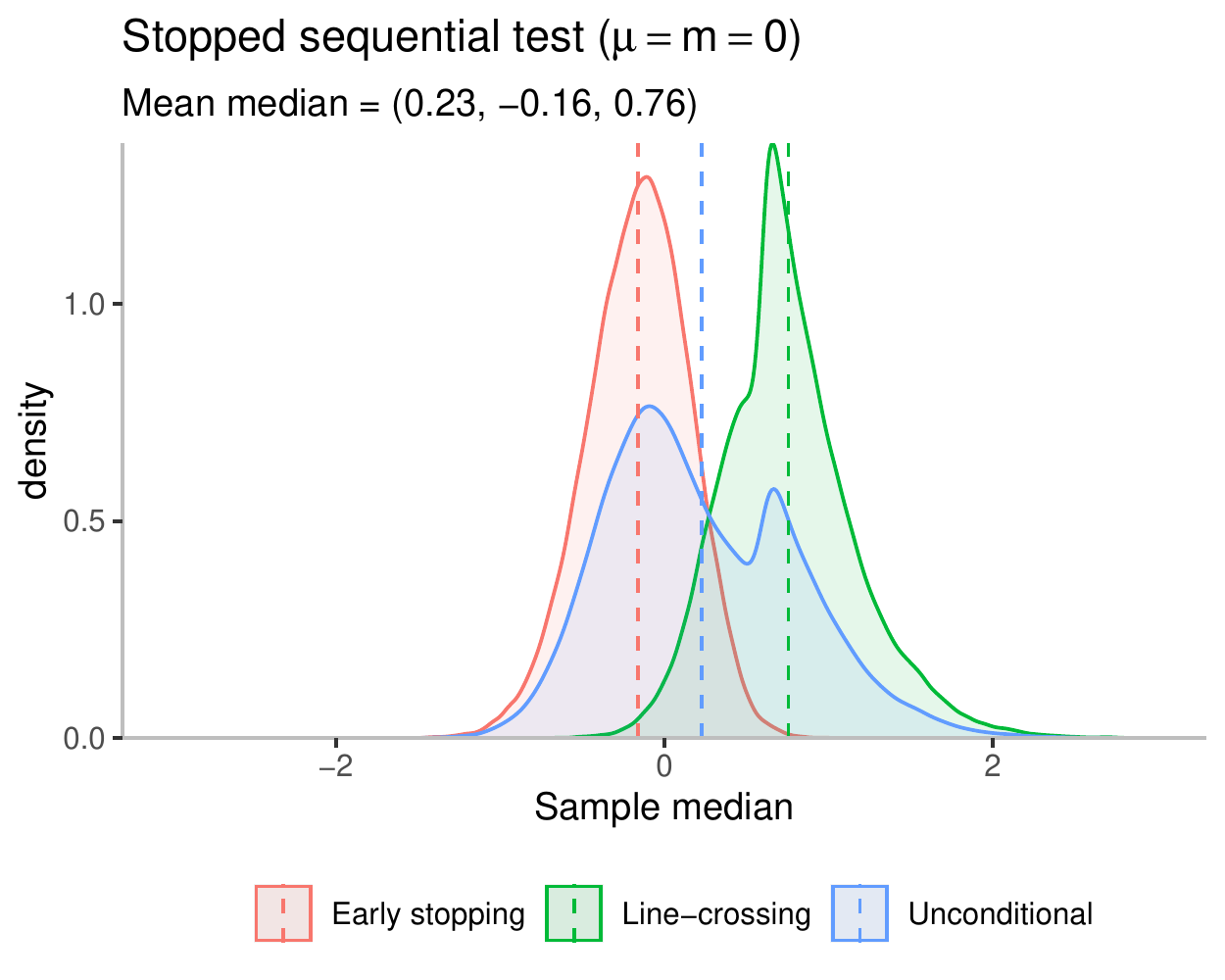}
    \end{center}
    \caption{\em Left: Densities of observed sample variances from repeated stopped sequential test as described in \cref{subSec::eg1}. Right: Densities of observed sample median from the same repeated stopped sequential test. For both figures, vertical dashed lines correspond to  averages of sample variances and medians on each conditions.}
    \label{fig::example1_var}
\end{figure}

As an initial step, we conduct the repeated sequential experiments described in \cref{subSec::eg1} and empirically investigate the biases of the sample variance and median estimators. \cref{fig::example1_var} describes a simulation study on the bias of the sample variance in the sequential testing setting of \cref{subSec::eg1}. Recall that, in this experiment, we have a stream of samples from a standard normal distribution. Each test terminates once either the number of samples reaches a fixed early stopping time $M = 10$ or the sample mean crosses the upper boundary $t \mapsto \frac{z_\alpha}{\sqrt{t}}$ with $\alpha = 0.2$. 

\cref{fig::example1_var} shows the marginal and conditional distributions of the sample variance and median from $10^5$ stopped sequential tests. Vertical lines correspond to averages of the sample variances and medians over repetitions of the experiment and under different conditions. For the sample variance, the simulation shows that the sample variance is negatively biased marginally and conditionally on the early stopping event. On the other hand, conditionally on the line-crossing event, the sample variance has a heavy right tail and is positively biased. For the sample median, we can check that, marginally and conditionally on the line-crossing event, the sample median is positively biased. In contrast, the sample median is negatively biased conditionally on the early stopping event. Note that, for the sample median, sizes of biases of the sample median are similar to ones from the sample means which were equal to $(0.22, -0.16, 0.75)$.

\end{document}